\documentclass[12pt]{amsart}

\usepackage{fullpage}
\usepackage{mathtools}
\usepackage{amsmath,amsfonts,amssymb,amsthm}
\usepackage[mathscr]{eucal}
\usepackage{tikz-cd}
\usepackage{enumitem,kantlipsum}
\usepackage[normalem]{ulem}
\usepackage[numbers]{natbib}
\usepackage{caption}
\usepackage{subcaption}
\usepackage{tikz}
\usepackage{todonotes}

\usetikzlibrary{shapes,positioning,intersections,quotes}

\usepackage{datetime2}

\RequirePackage[colorlinks,citecolor=blue,urlcolor=blue,backref=page]{hyperref}

\usepackage[final]{showlabels} 

\makeatletter
\@namedef{subjclassname@2020}{\textup{2020} Mathematics Subject Classification}
\makeatother

\makeatletter
\newtheorem*{rep@theorem}{\rep@title}
\newcommand{\newreptheorem}[2]{%
\newenvironment{rep#1}[1]{%
\def\rep@title{#2 \ref{##1}}%
\begin{rep@theorem}}%
{\end{rep@theorem}}}
\makeatother

\newtheorem{theorem}{Theorem}[section]

\newreptheorem{theorem}{Theorem}

\newtheorem{lemma}[theorem]{Lemma}
\newtheorem{assumption}[theorem]{Assumption}

\newreptheorem{lemma}{Lemma}

\newtheorem{definition}[theorem]{Definition}
\newtheorem{remark}[theorem]{Remark}

\newtheorem{prop}[theorem]{Proposition}

\newcommand{\beq}{\begin{eqnarray}}
\newcommand{\eeq}{\end{eqnarray}}
\newcommand{\beqq}{\begin{eqnarray*}}
\newcommand{\eeqq}{\end{eqnarray*}}
\def\Rar{\Rightarrow}
\def\rar{\rightarrow}

\newcommand{\BP}{\mathbb{P}}

\newcommand{\EXP}[1]{\mathbb{E}\!\left\{#1\right\} }

\newcommand{\VAR}[1]{\mathsf{VAR}\!\left(#1\right) }
\newcommand{\remove}[1]{}

\newcommand{\Hg}{\mathrm{H}}

\newcommand{\mR}{\mathbb{R}}

\newcommand{\bN}{\mathbb{N}}

\newcommand{\cN}{\mathcal{N}}

\newcommand{\mrD}{\mathring{\mathcal{D}}}
\newcommand{\mrC}{\mathring{\mathcal{C}}}

\newcommand{\mcF}{\mathcal{F}}
\newcommand{\mcK}{\mathcal{K}}
\newcommand{\mcL}{\mathcal{L}}
\newcommand{\mcC}{\mathcal{C}}
\newcommand{\mcD}{\mathcal{D}}
\newcommand{\pM}{\mathbf{M}^{\phi}}
\newcommand{\pMi}{\mathbf{M}}

\newcommand{\0}{\textbf{0}}

\newcommand{\pF}{\mathbf{F}^{\phi}}
\newcommand{\pB}{\mathbf{B}^{\phi}}
\newcommand{\pBi}{\mathbf{B}}

\newcommand{\pL}{\mathbf{L}}

\newcommand{\1}{\mathbf{1}}

\def\P{\mathcal P}

\newcommand{\X}{{\mathcal X}}
\newcommand{\Y}{{\mathcal Y}}

\newcommand{\Vor}{\mathrm{Vor}}
\DeclareMathOperator{\MSA}{MSA}
\DeclareMathOperator{\dgm}{Dgm}
\DeclareMathOperator{\diam}{diam}
\DeclareMathOperator{\MST}{MST}
\newcommand{\cone}{\Lambda}
\newcommand{\cpts}{\mathcal{Q}}

\begin{document}

\title{Central Limit Theorem for Euclidean Minimal Spanning Acycles}

\date{\today}

\author{Primoz Skraba} 
\address{School of Mathematical Sciences, Queen Mary University of London, London.}
\email{p.skraba@qmul.ac.uk}

\author{D. Yogeshwaran}
\address{Theoretical Statistics and Mathematics Unit, Indian Statistical Institute, Bangalore.}
\email{d.yogesh@isibang.ac.in}



\begin{abstract}
We investigate asymptotics for the minimal spanning acycles of the (Alpha)-Delaunay complex on a stationary Poisson process on $\mR^d, d \geq 2$. Minimal spanning acycles are topological (or higher-dimensional) generalization of minimal spanning trees. We establish a central limit theorem for total weight of the minimal spanning acycle on a Poisson-Delaunay complex. Our approach also allows us to establish central limit theorems for sum of birth times and lifetimes in the persistent diagram of the Delaunay complex. The key to our proof is in showing the so-called {\em weak stabilization} of minimal spanning acycles which proceeds by establishing suitable chain maps and uses matroidal properties of minimal spanning acycles. In contrast to the proof of weak-stabilization for Euclidean minimal spanning trees via percolation-theoretic estimates, our weak-stabilization proof is algebraic in nature and provides an alternative proof even in the case of minimal spanning trees. 
\end{abstract}

\maketitle

\section{Introduction}
\label{s:intro}

Let $\P_n = \{X_1,\ldots,X_{N_n}\}$ be the restriction of a stationary Poisson point process $\P$ with intensity $\lambda > 0$ to the window $W_n := \left[-\frac{n^{1/d}}{2},\frac{n^{1/d}}{2}\right]^d$. Equivalently, $N_n$ is a Poisson($\lambda n$) random variable and $X_i$'s are i.i.d. uniform points in $W_n$ independent of $N_n$. Consider the complete graph on $\P_n$ with edge-weights given by the Euclidean distance between the points. A classical problem at the intersection of combinatorial optimization and geometric probability is to investigate the total edge-lengths (or total weighted edge-length) of the minimal spanning tree on this weighted graph. Strong laws for this statistic and many such Euclidean optimization functionals were proven using subadditive ergodic theory methods; see \citet{Steele97,Yukich98}. For the case of sum of power-weighted edge-lengths, i.e., edge-lengths raised to $d$-th power, a strong law was also proven using local weak convergence ideas; see \citet{aldous1992asymptotics}. The central limit theorem for the same was remarked to be a difficult problem in \cite{aldous1992asymptotics}. It was solved independently by \citet{Alexander1995} and \citet{kesten1996central}. The former used an approach proposed by \citet{ramey1983non} in $d = 2$ while the latter used a martingale-difference central limit theorem which worked in all dimensions and for sum of weighted edge-lengths as well. Both  proofs, however,  use certain percolation-theoretic estimates. The aim of this article is to prove a central limit theorem (CLT) for a topological generalization of minimal spanning trees called  minimal spanning acycles on weighted Poisson-Delaunay complexes. The main challenge in  this  extension is the lack of percolation-theoretic analogue in higher dimensions. Strong laws for statistics of minimal spanning acycles follow from the results of \cite{divol2019choice} where sub-additive ergodic theory methods are used. \\

The rest of the article is organized as follows. In Section \ref{sec:Main_theorem}, we briefly introduce our model  and state our main results. We also discuss our proof strategy and compare our results with those in the existing literature. In Section \ref{s:prelims}, we introduce the necessary topological and probabilistic preliminaries. All our proofs are contained in Section \ref{s:proofs}. As it may be of independent interest and is more easily understandable, we give an overview of our proof for case of minimal spanning trees in Section \ref{s:proof_MST}. At the beginning of each section, we give a more detailed description of the subsections therein. Though we shall introduce all the necessary topological notions to understand all our theorem statements, some of the proofs (in particular those in Section \ref{s:stabchainmaps} and \ref{s:lem_del}) shall assume some basic knowledge of algebraic topology. 

\subsection{Model and main results}
\label{sec:Main_theorem}
For more formal definition of (simplicial) complexes, minimal spanning acycles and their relation to homology, we refer the reader to Section \ref{s:prelims}. Here, we will give a formal definition of Delaunay complexes but shall define minimal spanning acycles informally. This shall aid in an easier presentation of results.
\begin{figure}[tbp]
\centering\includegraphics[width=0.6\textwidth]{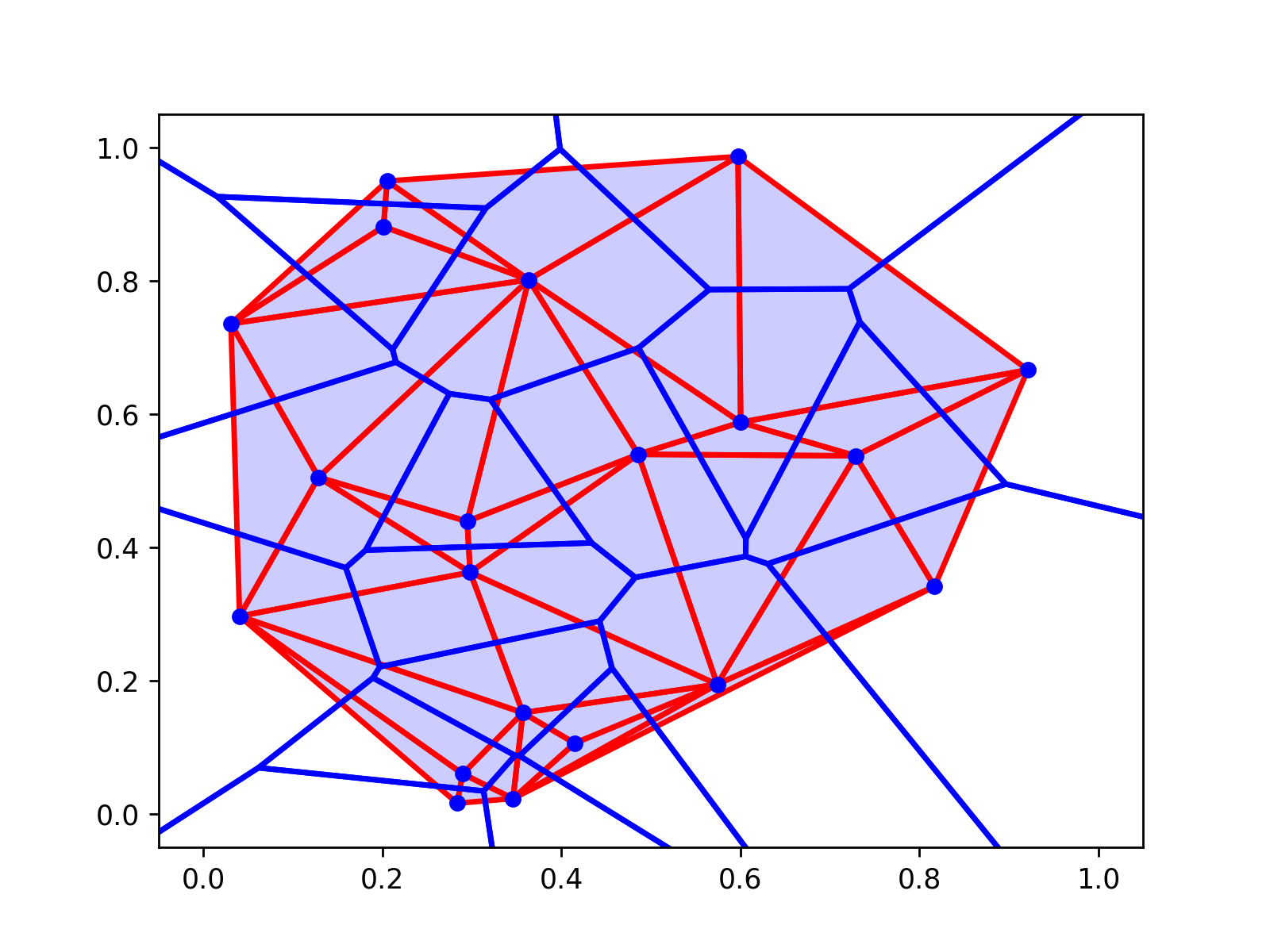}
\caption{\label{fig:voronoi} The Voronoi diagram (boundary of Voronoi cells are in blue) on a point-set (Poisson process) in $\mathbb{R}^2$ and the dual Delauany complex (with edges in red).}
\end{figure} 
\subsubsection*{\bf Delaunay Complex.} 
 Let $\X \subset \mR^d$ be a locally-finite point-set with points in general position. Recall that general position means that no three points are colinear and no $d+2$ points lie on a common sphere, i.e., are cocircular. We define the {\em Voronoi cell} of $x \in \X$ as
$$ \Vor_{\X}(x) := \{ y : |y-x| \leq |y-x'|, \forall x' \in \X \},$$
i.e., the set of points in $\mR^d$ whose closest point in $\X$ is $x$. Each $\Vor_{\X}(x)$ is a convex polyhedron and the collection $\{\Vor_{\X}(x)\}_{x \in \X}$ forms a partition of $\mR^d$ with disjoint interiors known as the {\em Voronoi diagram or tessellation}. See Figure~\ref{fig:voronoi} for an example.   

\begin{definition}[Delaunay Complex]
\label{def:delcomp}
Let $\X \subset \mR^d$ be a locally-finite point-set with points in general position. The {\em Delaunay complex} on $\X$, denoted as $\mcD(\X)$ is the following simplicial complex:
$$ \mcD(\X) := \cup_{k=0}^{\infty}\{ [x_0,\ldots,x_k] : \cap_{i=0}^k\Vor_{\X}(x_i) \neq \emptyset \}.$$
The element $[x_0,\ldots,x_k]$ is called a $k$-face of the Delaunay complex and the collection of $k$-faces of the Delaunay complex is denoted by $\mcF_k(\X)$. Generic faces of $\mcD(\X)$ is denoted by $\sigma, \tau$ et al. The Delaunay graph on $\X$ is the graph with vertex set $\X$ and edge-set given by $1$-faces. 
We define the weight function on the Delaunay complex as follows :
$$ w([x_0,\ldots,x_k]) = \inf \{ s : \cap_{i=0}^k B_s(x_i) \cap \Vor_{\X}(x_i) \neq \emptyset \}$$
where $B_s(x_i)$ denotes the ball of radius $s$ centered on $x_i$.
\end{definition}
It is a standard fact that if points are in general position, the Delanauy complex and the Voronoi diagram are dual to one another and that each simplex in the Delaunay complex corresponds to an empty circumsphere containing its vertices, with the weight being the radius of this sphere. The above weighted (simplicial) complex is also known as the Alpha-Delaunay complex. A closely related weight function 
$$ w'([x_0,\ldots,x_k]) = \inf \{ s : \cap_{i=0}^k B_s(x_i) \neq \emptyset \}, \, [x_0,\ldots,x_k] \in \mcD(\X)$$
gives rise to the Delaunay-\v{C}ech complex; see~\citet[Section 3]{bauer2017morse}. Our results apply to Delaunay-\v{C}ech complex as well but we restrict to Alpha-Delaunay complex for convenience, referring to it simply as the Delaunay complex.

The dimension of the Delaunay complex is the dimension of the largest face, which under the general position assumption, is the dimension of the ambient space - $d$. 

\subsubsection*{\bf Minimal spanning acycles.}
\begin{figure}
	\centering
	\begin{subfigure}[b]{0.3\textwidth}
		\centering
		\includegraphics[width=\textwidth]{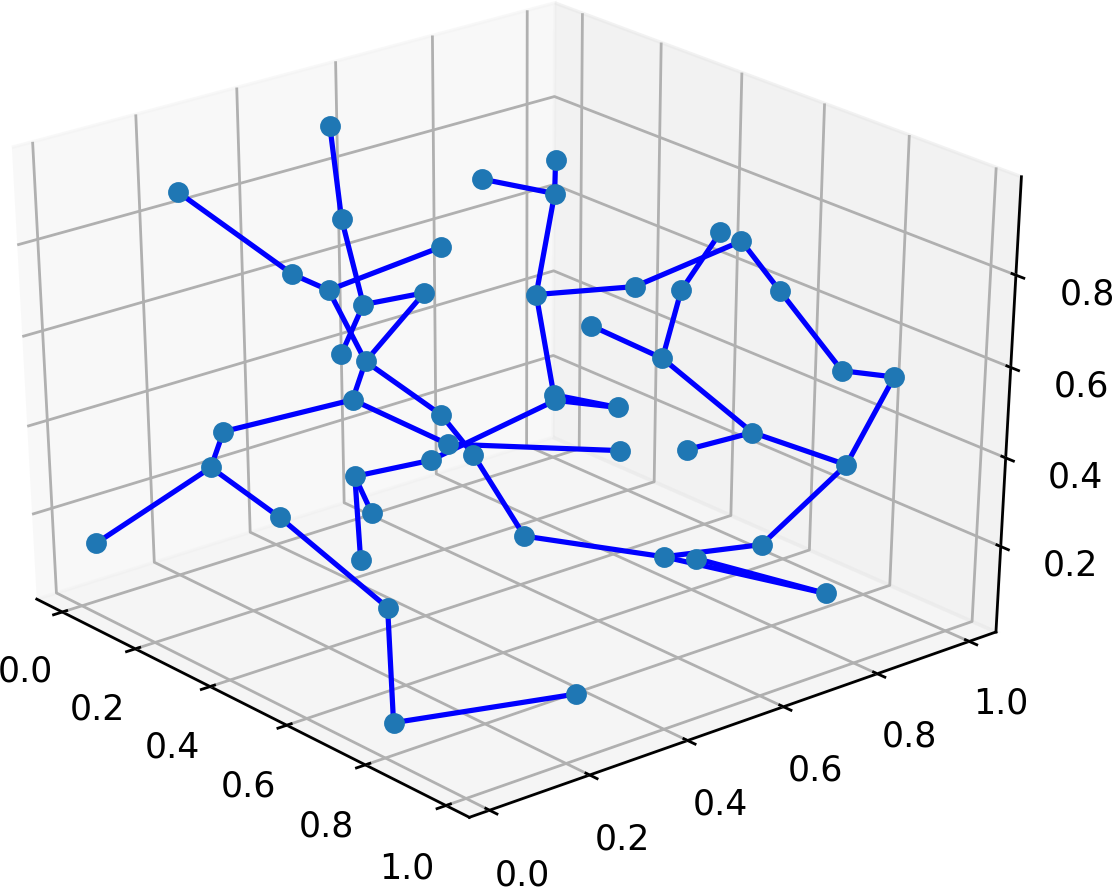}
		\caption{\label{fig:mst}}
	\end{subfigure}
	\hfill
	\begin{subfigure}[b]{0.3\textwidth}
		\centering
		\includegraphics[width=\textwidth]{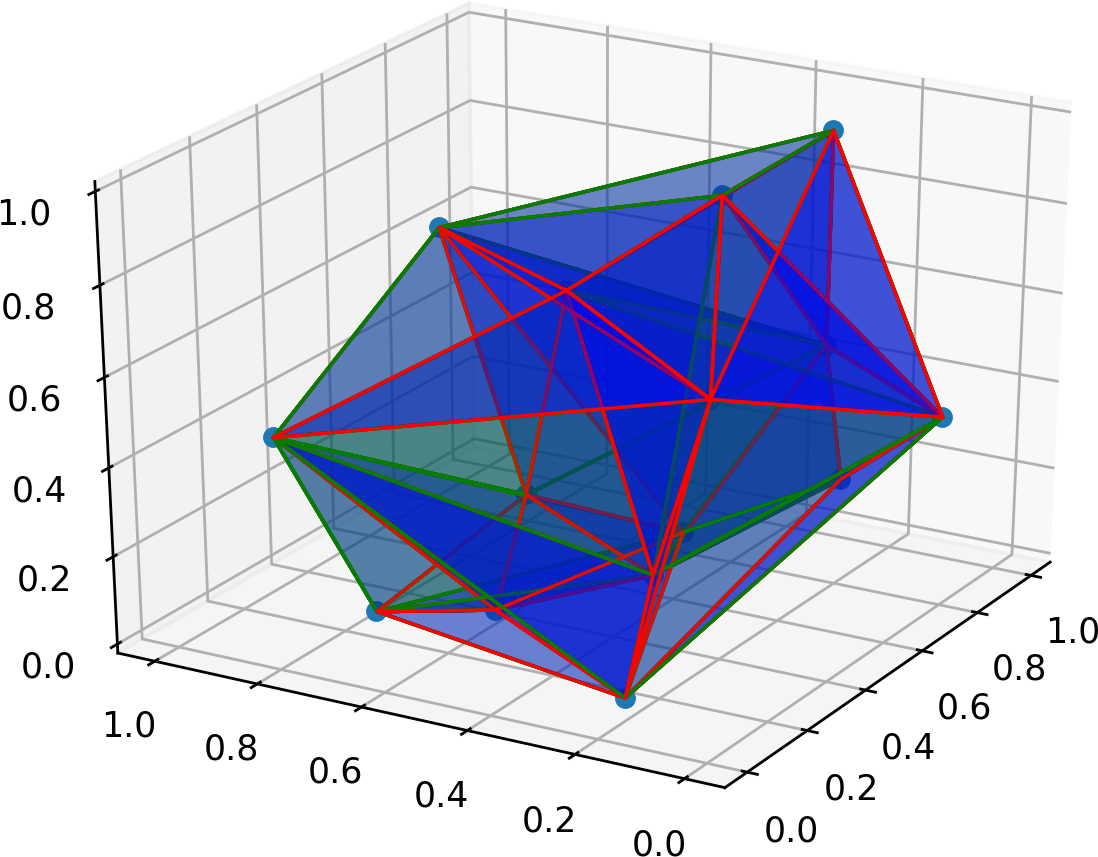}
		\caption{\label{fig:msa1}}
	\end{subfigure}
	\hfill
	\begin{subfigure}[b]{0.3\textwidth}
		\centering
		\includegraphics[width=\textwidth]{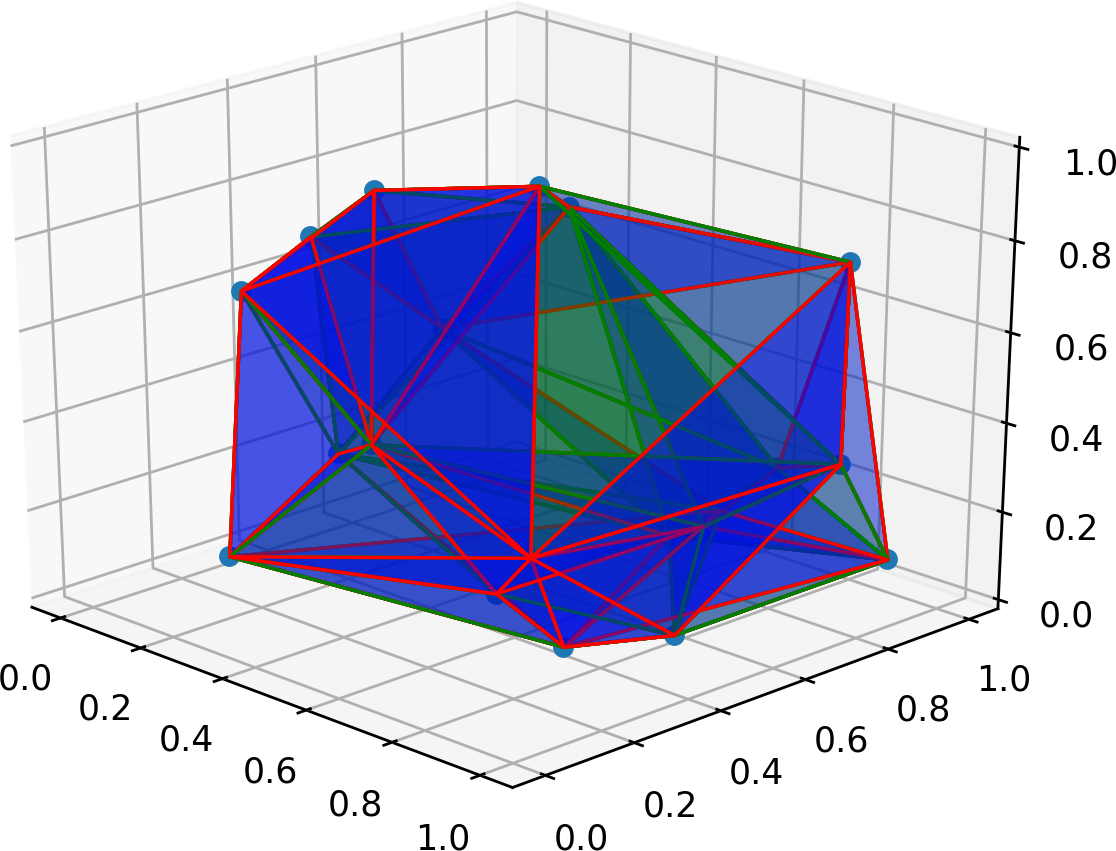}
		\caption{\label{fig:msa2}}
	\end{subfigure}
	   \caption{ In  $\mathbb{R}^3$ (A) a minimum spanning tree  and (B,C) two examples of MSAs.  In the latter, the blue faces are in MSA and the green faces are not in MSA. Both the figures correspond to Delaunay complexes on the Poisson process.}
	   \label{fig:msas}
\end{figure}
We now introduce spanning acycles and minimal spanning acycles (MSA) on a Delaunay complex informally. Set the $0$-spanning acycle to be empty set. A {\em $1$-spanning acycle} on the Delaunay complex is simply a spanning tree on the Delaunay graph. Recall that a spanning tree is a collection of edges that does not contain any cycles, i.e., does not enclose any $1$-dimensional holes, and is connected. See Figure \ref{fig:msas}(A) for simulation of a minimal spanning tree on a Delaunay graph when $\X$ is a Poisson process in $\mR^3$. A {\em $k$-spanning acycle} is a collection of $k$-faces of the Delaunay complex that do not enclose $k$-dimensional holes and also fills in (i.e. bounds) all possible $(k-1)$-dimensional holes enclosed by all the $(k-1)$-faces, i.e., $\mcF_{k-1}(\X)$. For example, when $k = d$, one cannot create any $d$-dimensional holes and to fill in all $(d-1)$-dimensional holes, all the $d$-faces must be included and so the spanning acycle is $\mcF_d(\X)$. 
%
%
Given a spanning acycle (or any collection of faces), one can define the weight of the spanning acycle to be the sum of weights of faces in the spanning acycle i.e., if $M \subset \mcF_k(\X)$ is a $k$-spanning acycle, then 
\[ w(M) := \sum_{\sigma \in M} w(\sigma).\]
Immediately, one can define {\em a minimal spanning acycle} as a spanning acycle with minimum weight. If the weights on faces (of a fixed dimension $k$) are unique, which is the case when $\X$ is in general position, then there is a unique minimal $k$-spanning acycle (MSA). When unique, we denote the MSA on $\mcD(\X)$ by $MSA(\X)$ with the dependence on $k$ suppressed for convenience. Figures ~\ref{fig:msas}(B,C) shows two simulations of $2$-MSAs when $\X$ is a Poisson process in $\mR^3$.

Though the concept of spanning acycles was introduced by \citet{kalai1983enumeration} in 1983, minimal spanning acycles have only received attention in recent years \cite{edelsbrunner2020tri} and especially in the context of random simplicial complexes \cite{hiraoka2017minimum,Skraba17,hino2019asymptotic,kanazawa2021law}.

\subsubsection*{\bf Poisson-Delaunay Complexes.} The object of interest to us are MSAs on the Poisson-Delaunay complex. We take $\X = \P$, the stationary Poisson point process of intensity $\lambda > 0$ and consider the Delaunay complex $\mcD(\P)$. More specifically, we set $\P_n := \P \cap W_n$, the restriction of the stationary Poisson process to the window $W_n = [\frac{-n^{1/d}}{2}, \frac{n^{1/d}}{2}]$ and investigate the asymptotic behaviour for $k$-minimal spanning acycles on the sequence of Delaunay complexes $\mcD(\P_n)$. Since the face-weights for a fixed dimension $k$ are a.s. distinct, we have that $\MSA(\P_n) := \MSA(\mcD(\P_n))$ is a.s. unique by \citet[Lemma 25]{Skraba17}. The statistic of interest for us is {\em the sum of face-weights} defined as
\begin{equation}
\label{e:mkphi}
\pM_k(\P_n) = \pM_k(\mcD(\P_n)) := \sum_{\sigma \in \MSA(\P_n)} \phi(w(\sigma)),
\end{equation}
where $\phi : \mR_+ \to \mR_+$ is a strictly increasing function. We denote $\pM_k$ by $\pMi_k$ when $\phi$ is identity. From the relation between persistent homology and minimal spanning acycles (see Theorem \ref{t:birthdeathmsa}), we know that $\pMi_k(\P_n)$ is the sum of death times in the persistence diagram of $\mcD(\P_n)$. With this connection, it is natural also to study sum of birth times as well a sum of lifetimes in a persistence diagram. We define them as
\begin{align}
\label{e:bklkphi}
\pB_k(\P_n) &:= \pB_k(\mcD(\P_n)) = \smashoperator{\sum_{ \sigma \in \mcF_k(\P_n) \setminus MSA(\P_n) }} \phi(w(\sigma)) \\
\pL_k(\P_n) &:= \pL_k(\mcD(\P_n)) = \pMi_k(\P_n) - \pBi_{k-1}(\P_n). 
\end{align}
Note that trivially $B_0(\P_n) = 0$ as $\mcF_0(\P_n) = \P_n, \MSA_0(\P_n) = \emptyset$ and $w(v) = 0$ for all $v \in \P_n$. Thus, $\pL_1(\P_n) = \pMi_1(\P_n)$ (the total edge-length of minimal spanning tree). 

Strong laws for $\pB_k(\P_n)$ and $\pL_k(\P_n)$ can be deduced from strong laws for $\pM_k(\P_n)$ and $\pF_k(\P_n)$. Strong laws and central limit theorem for $\pF_k(\P_n) := \sum_{\sigma \in \mcF_k(\P_n)}\phi(w(\sigma))$ can be deduced from \citet[Theorem 3.2]{penrose2002limit} and \citet[Theorem 2.1]{Penrose01}. Strong laws for $\pM_k(\P_n)$ follow from \citet[Theorem 6]{divol2019choice}. Hence, we focus on central limit theorems for $\pB_k(\P_n)$, $\pM_k(\P_n)$ and $\pL_k(\P_n)$.

\subsubsection*{\bf Weak Stabilization and central limit theorem.} Our first main result is to show a weak stabilization result for sum of death times, birth times and lifetimes of  persistent diagrams of weighted Poisson-Delaunay complexes. 

Let $A_n$ be sequence of boxes (i.e., $A_n = W_m + x$ for some $m \geq 1$ and $x \in \mR^d$) such that $A_n \to \mR^d$; see Section \ref{sec:weak_stab} for precise definition of boxes ($\mathfrak{A}$) and convergence of sets. 
\begin{prop}
\label{prop:stab_death_times}
Let $1 \leq k \leq d$ and $\phi : \mR_+ \to \mR_+$ be a strictly increasing function. Let $A_n$ be defined as above and $\P_n = \P \cap A_n$. All three statistics - $\pM_k(\P_n), \pB_k(\P_n)$, and $\pL_k(\P_n)$ - satisfy weak stabilization as in \eqref{eqn:weakstab-rob} i.e., as $n \to \infty$
\begin{equation}
\label{e:WstabBn}
D_{\0} \pM_k(\P_n) := \pM_k(\P_n \cup \{0\}) - \pM_k(\P_n) \to D_{\infty}(\pM_k), \, \, \mbox{a.s.}, \quad 
\end{equation}
where $D_{\infty}(\pM_k)$ is an a.s. finite random variable and similarly for $\pB_k(\P_n)$ and $\pL_k(\P_n)$.
\end{prop}
\citet{kesten1996central} introduced the the notion of weak stabilization to prove a central limit theorem for Euclidean minimal spanning trees and this was abstracted to the above form of weak stabilization for general Poisson functionals in \citet{Penrose01}. This was further improved by \citet{trinh2019central}; see Theorem \ref{thm:clt_Duy17}. We now state our central limit theorem for the sum of birth times, death times and lifetimes. We will need to impose a growth assumption on the weight function $\phi$ to verify certain moment conditions.
\begin{theorem}
\label{thm:clt_msa}
Let $\phi : \mR_+ \to \mR_+$ be a strictly increasing function such that $\phi(t) \leq Ct^p, t \geq 0$ for some $p > 0$. Then, there exist constants $ \sigma^2(\pMi_k,\phi) \in (0,\infty), 1 \leq k \leq d$ and $\sigma^2(\pBi_k,\phi), \sigma^2(\pL_k) \in (0,\infty),$ for $1 \leq k \leq d-1$ such that the following hold. As $n \to \infty$, for $1 \leq k \leq d$,
\[ n^{-1}\VAR{\pM_k(\P_n)} \to \sigma^2(\pMi_k,\phi), \, \, n^{-1/2}\left[\pM_k(\P_n) - \EXP{\pM_k(\P_n)}\right]\ {\Rar}\ N(0,\sigma^2(\pMi_k,\phi)),\] 
and for $1 \leq k \leq d-1$,
\[ n^{-1}\VAR{\pB_k(\P_n)} \to \sigma^2(\pBi_k,\phi), \, \, n^{-1/2}\left[\pB_k(\P_n) - \EXP{\pB_k(\P_n)}\right]\ {\Rar}\ N(0,\sigma^2(\pBi_k,\phi)), \]
and
\[ n^{-1}\VAR{\pL_k(\P_n)} \to \sigma^2(\pL_k), \, \, n^{-1/2}\left[\pL_k(\P_n) - \EXP{\pL_k(\P_n)}\right]\ {\Rar}\ N(0,\sigma^2(\pL_k)), \]
where $\Rar$ denotes convergence in distribution and $N(0,\sigma^2)$ denotes the Normal random variable with mean $0$ and variance $\sigma^2 \in [0,\infty)$.
\end{theorem}

For $k = 1$, since $\pM_1$ is the $\phi$-weighted sum of edges on a minimal spanning tree, weak stabilization and central limit theorem follow from the results of \citet{kesten1996central} (see also \citet[Proposition 1]{lee1997central}, \citet[Proposition 1]{lee1999central} and \citet[Lemma 2.1]{penrose2003weak}), but these proofs use percolation properties of the the Poisson-Boolean model $\cup_{X \in \P}B_r(X), r \geq 0$. With the study of homological percolation still in a nascent stage (see \citet{skraba2020homological}), proof of weak stabilization in higher-dimensions necessitates a different approach. Our proof relies mainly upon local-finiteness of the weighted Delaunay graph, existence of certain homomorphisms between chain complexes and also implicitly that minimal spanning acycles can be constructed via greedy algorithms. We discuss the proof in more detail at the end of this section. Also to follow our proof ideas better, we give a proof overview in the case of minimal spanning tree in Section \ref{s:proof_MST} where we consider $\mcD \cap \P_n$ instead of $\mcD(\P_n)$. The difference between $\mcD \cap \P_n$ and $\mcD(\P_n)$ lies in the boundary effects. While this doesn't significantly affect the percolation theoretic approach, our approach necessitates the use of suitable chain homomorphisms to overcome the boundary effects. In the case of $\mcD \cap \P_n$, one can use combinatorial arguments to deduce stabilization of the MSA,  but for $\mcD(\P_n)$, simplices may disappear with increasing window size. So rather than rely on monotoncity, we resort to the existence of certain chain maps (Section \ref{s:lem_del}), which requires  more algebraic arguments (see Section \ref{s:stabchainmaps}). 

The main advantage of our approach is that it requires no change when going from minimal spanning trees to minimal spanning acycles. We use the randomness of the point process in controlling the stabilization of $\mcD$. This can be extended to fairly general point processes, see \citet[Theorem 2.5]{BYY2019} and the discussion therein. However, the central limit theorem which relies heavily on Theorem \ref{thm:clt_Duy17}, does require the Poisson point process assumption. 

We now comment on the use of weak stabilization in random topology and connections to our result. Weak stabilization has been a crucial tool to prove central limit theorems for topological statistics in recent years. This is mainly because other stronger notions of stabilization are insufficient to handle long-range dependence of topological statistics especially in the presence of percolation in the Poisson-Boolean model. In \citet{Yogesh17}, weak stabilization was shown for Betti numbers and later this was extended to persistent Betti numbers in \citet{Hiraoka2018limit, krebs2019asymptotic}. The statistics of minimal spanning acycle have more global dependence than Betti numbers and persistent Betti numbers and hence the methods do not naturally extend to minimal spanning acycles. The relation between Betti numbers and total lifetime sum ($\pL_k$) can be exploited to prove weak stabilization of $\pL_k$ as in the case of random cubical complexes in \citet{Hiraoka2018cub}. But this does not extend to the weight of a minimal spanning acycle (equivalently death times in the persistence diagram). In all these works, weak stabilization was shown en-route to the proof of central limit theorem for these statistics. Proposition \ref{prop:stab_death_times} and Theorem \ref{thm:clt_msa} significantly extend these results by proving stabilization not only for lifetime sums but also more generally for weight of a minimal spanning acycle and its complement. Recently strong stabilization has been shown for persistent Betti numbers in \cite{krebs2019asymptotic} but it is unclear whether this can be extended to minimal spanning acycles.

We now mention some possible extensions and related results in non-geometric or mean-field models. Quantifying the convergence in \eqref{e:WstabBn} via percolation-theoretic estimates for the Poisson-Boolean model has been used to prove normal approximation results for Euclidean minimal spanning trees in \citet{chatterjee2017minimal,lachieze2020quantitative}. Though one would expect this to  hold for minimal spanning acycles, a proof is far from obvious at this stage. The results of \cite{krebs2019asymptotic} may be considered a first step towards such a quantification and in also extending  central limit theorems in \cite{Yogesh17,Hiraoka2018limit} to inhomogeneous Poisson processes and Binomial point processes. 

In a different direction, central limit theorem for minimal spanning trees on randomly weighted complete graphs was proven in \citet{janson1995minimal}. Strong laws for lifetime sums have been shown for simplicial complex versions of randomly weighted complete graphs in \citet{hino2019asymptotic} using the relation between Betti numbers and lifetime sums. A central limit theorem still evades us in these models and again the stabilization-based approach of \citet{cao2021central} offers promise. 

\subsubsection*{\bf Discussion of proofs:} We end the introduction with a detailed sketch of our proofs. As is obvious from the above discussions, the key result is Proposition \ref{prop:stab_death_times} and then applying Theorem \ref{thm:clt_Duy17}, we obtain the central limit theorem - Theorem \ref{thm:clt_msa}. The crucial step towards the proof of Proposition \ref{prop:stab_death_times} is in understanding stabilization when $A_n = W_n$ i.e., a sequence of centred boxes. Proposition \ref{prop:stabilization_D} proves stabilization in this case and we shall outline here this proof alone. 

One of the technical difficulties in proving stabilization along windows $W_n$ (i.e., Proposition \ref{prop:stabilization_D}) lies in overcoming the boundary effects in the Delaunay complex. More precisely, it is easier to show weak stabilization for $\mcD(\P) \cap W_n$ but what is required is weak stabilization of $\mcD(\P_n)$. The former sequence of complexes satisfy monotonicity in $n$ but the latter do not. We overcome this via by passing to the nerve of covers induced by the Voronoi diagrams and using chain maps between $\mcD(\P_n)$ and $\mcD(\P_m)$ for $n < m$. 

The starting point of our proof strategy is to decompose addition of the origin into a.s. a finite number of changes in the complex locally around the origin. By showing that the effect of these elementary changes stabilizes for Delaunay complexes, we show that the effect of adding the origin stabilizes. More precisely, denoting by $\mcD$ the Delaunay complex on $\P$ and $\mrD$ the Delaunay complex on $\P \cup \{\0\}$, we show the following statements (either a.s. or pointwise):
\begin{enumerate}
\item The removal and addition of simplices to go from $\mcD$ to $\mrD$ happens in $W_N$ for some $N$ (random but a.s. finite); Proposition \ref{prop:const_diff}. 

\item Thus for $m > N$, $\pM_k(\mcD(P_m \cup \{\0\})) - \pM_k(\mcD(\P_m))$ can be expressed as sum of finite differences of the form $\pM_k(\mcD'_m) - \pM_k(\mcD_m)$ such that $\mcD'_m, \mcD_m$ are obtained from $\mcD(\P_m)$ or $\mcD(P_m \cup \{\0\})$ by the removal of finitely many simplices and furthermore $\mcD'_m$ and $\mcD_m$ differ by at most one simplex. 

\item Though there is no inclusion between $\mcD'_m, \mcD'_n$ for $n > m$ (resp.  $\mcD_m, \mcD_n$),  we show existence of appropriate chain maps between $\mcD'_m$ and $\mcD'_n$
(resp. between $\mcD_m$ and $\mcD_n$); see Section \ref{s:lem_del}.

\item Then, using a series of lemmas (Section \ref{s:stabchainmaps}) that analyse stabilization under chain maps when one simplex is added or removed, we show in Proposition \ref{prop:stabilization_D} that $\pM(\mcD'_m) - \pM(\mcD_m)$ converges as $m \to \infty$.

\item To show that the limiting variance is non-degenerate,  we show that limit of $\mcD(P_m \cup \{\0\}) - \mcD(\P_m)$ is a non-zero random variable by constructing some specific configurations for which the limiting random variable is strictly greater than zero with positive probability; see Section \ref{s:varlb}.
\end{enumerate}
For ease of understanding the main points of the above proof strategy and for the benefit of the readers interested in the case of minimal spanning trees, we give a 
proof overview for the special case of $k = 1$ and where $\phi$ is identity 
in Section \ref{s:proof_MST}.

Also, as mentioned below Definition \ref{def:delcomp}, our proof also works for Delaunay-\v{C}ech complex. All the above steps follow straightforwardly for the Delaunay-\v{C}ech case except that of variance lower bound proof. The latter uses specific constructions which must be suitably modified (see Remark \ref{r:configdc}).

\section{Preliminaries}
\label{s:prelims}

Weighted complexes, spanning acycles and connection to persistent homology are introduced respectively in Sections \ref{sec:weightsimpcompl}, \ref{sec:msaprelim} and \ref{sec:pers_homology}. Section \ref{sec:weak_stab} introduces Poisson process and recalls the general central limit theorem for functionals of Poisson process from \citet{trinh2019central}.

For more on algebraic topology notions considered in this paper, we refer the reader to~\citet{Edelsbrunner10,Munkres84,Hatcher02}.

\subsection{Weighted complexes}
\label{sec:weightsimpcompl}

We introduce weighted (simplicial) complexes and the associated filtration. A {\em (simplicial) complex} $\mcK$ on a vertex set $V$ is a collection of subsets of $\mcK$ such that if $\sigma \in \mcK$ and $\tau \subset \sigma$ then $\tau \in \mcK.$ We shall assume that $V \subset \mcK$. Elements of $\mcK$ are called {\em faces or simplices} and faces of cardinality $(k+1)$ are called as $k$-dimensional faces or simplices. We represent a $k$-dimensional face $\tau$ by $[v_0,\ldots,v_k]$ where $v_i \in V$, $i = 0,\ldots,k$. A {\em weighted complex} is a complex equipped with a weight function $w: \mcK \rightarrow \mR$. We shall assume that all our weight functions on a simplicial complex are {\em monotone} i.e., for all $\alpha \in \mR$, the preimage $w^{-1}(-\infty,\alpha]$ is a simplicial complex. The sub-complex of $\mcK$ consisting of only $k$-faces and lower is called {\em the $k$-skeleton} $\mcK_k$ i.e., $\mcK_k = \cup_{i=0}^k \mcF_i(\mcK)$. 
Given a weighted complex $\mcK$, we can define the following two filtrations :
\[ \mcK(t) := \{ \sigma \in \mcK : w(\sigma) \leq t \}, \, \, \mcK(t^-) := \{ \sigma \in \mcK : w(\sigma) < t \}, \, \, t \geq 0. \]
Also, given two weighted complexes $\mcK, \mcL$ with respective weight functions $w,w'$, we say $\mcK \subset \mcL$ if $\mcK$ is a subcomplex of $\mcL$ (i.e, the simplices of $\mcK$ are also in $\mcL$) and $w(\tau) = w'(\tau)$ for all $\tau \in \mcK$. For a weighted complex $\mcK$ and $\sigma$, a simplex, we denote by $\mcK(\sigma^-) = \mcK(w(\sigma)^-).$

Given a simplicial complex $\mcK$ and $k \geq 0$, let $C_k(\mcK), B_k(\mcK)$ and $Z_k(\mcK)$ denote the {\em $k$th-chain group}, {\em $k$-th boundary group} and {\em $k$-th cycle group} respectively\footnote{It will be clear from context whether the boundary group or a ball is denoted by $B$.}. These are defined via the boundary operator $\partial_k:C_k (\mcK)\rightarrow C_{k-1}(\mcK)$, where $Z_k(\mcK)$ the kernel of $\partial_k$ and $B_k$ is the image of $\partial_{k+1}$. The definition for $\partial$ is standard and we do not repeat it here. Also it is customary to suppress the subscript $k$ and we shall do so in future.

Set $C_{-1}(\mcK) = \mathbb{F}$, where $\mathbb{F}$ is our coefficient field. Since our coefficients are over a field $\mathbb{F}$, all our groups are actually $\mathbb{F}$-vector spaces. We define the {\em (reduced) $k$-th homology group $\Hg_k(\mcK)$} as the quotient group $\Hg_k(\mcK) := Z_k(\mcK) / B_k(\mcK).$ The rank of $\Hg_k(\mcK)$ is called as {\em the $k$-th Betti number} $\beta_k(\mcK)$. 

%
%

\subsection{Minimal spanning acycles}
\label{sec:msaprelim}
For a more detailed exposition on spanning acycles, we refer the reader to \citet[Section 2.1]{Skraba17}.

We now introduce spanning acycles. A subset $S \subset \mcF_k$ is said to be a {\em $k$-spanning acycle} if $\beta_{k-1}(\mcK_{k-1} \cup S) = \beta_{k-1}(\mcK)$ and $\beta_k(\mcK_{k-1} \cup S) = 0$. A {\em minimal $k$-spanning acycle} is a spanning acycle $S$ that minimizes the total weight $w(S) = \sum_{ \sigma \in S}w(\sigma)$. When the index $k$ in consideration is clear, we shall drop the index and refer to as spanning acyle and minimal spanning acycle (MSA). When unique, we shall denote the minimal spanning acycle by $\MSA$. The MSA is unique if the weight function $w$ is injective on $k$-faces \cite[Lemma 31]{Skraba17}. We refer the reader to Figure \ref{fig:msas} for illustrations of $2$-MSAs.

\subsection{Relation to persistent homology}
\label{sec:pers_homology}
We shall now relate minimal spanning acycles to persistent homology. This is important in understanding the implications of our results to persistent homology 
and also this shall yield us the important notion of negative and positive faces that will characterize minimal spanning acycles.

{\em Persistent homology} is an algebraic invariant which keep track of changes of homology in $\Hg_k(\mcK(t))$ for a weighted complex $\mcK$. \cite{Edelsbrunner10,Carlsson09,Carlsson14}. More formally, we define the {\em $(s,t)$-persistent homology group} as the quotient group 
$$ \Hg^{s,t}_k(\mcK) = \frac{Z_k(\mcK(s))}{Z_k(\mcK(t)) \cap B_k(\mcK(t))}, \quad s \leq t.$$
The ranks of $\Hg^{s,t}_k(\mcK), 0 \leq s \leq t \leq \infty$ can be encoded as a collection of intervals of $\mR_+$ called as {\em persistence barcode} \cite{zomorodian2005computing} or alternatively as a multi-set of points in called {\em the persistence diagram} $\dgm(\mcK,w)$ \cite{cohen2007stability}. We do not define persistence diagrams formally here but refer to \cite[Definition 9]{Skraba17}. For our purposes, it suffices to understand that $\dgm(\mcK,w)$ is a finite subset of $\{(s,t) : 0 \leq s \leq t \leq \infty \}$ and if $(b,d) \in \dgm(\mcK,w)$ it signifies that a certain homology class was `born' in $\mcK(b)$ and 'died' at $\mcK(d)$. The projection of the persistence diagram to the $y$-axis is called {\em the death times} (denoted by $\dgm_d$) and the projection to the $x$-axis is called {\em the birth times} (denoted by $\dgm_b$). There are two important alternate characterizations of birth and death times for us. The first is a straighforward one via Betti numbers and the second will be via minimal spanning acycles. 
\begin{lemma}(\cite[Section 3]{delfinado1993incremental})
\label{l:bettideath}
Let $\mcK$ be a finite weighted complex. If $\sigma \in \mcF_k$ then $\beta_j(\mcK(\sigma^-) \cup \{\sigma\}) = \beta_j(\mcK(\sigma^-))$ for $j \neq k, k -1$. Further, only one of the following two statements hold :
\begin{enumerate}[label=(\roman*),wide,  labelindent=0pt]
\item (Negative face) $\beta_{k-1}(\mcK(\sigma^-) \cup \{\sigma\}) = \beta_{k-1}(\mcK(\sigma^-)) - 1$ and $\beta_k(\mcK(\sigma^-) \cup \{\sigma\}) = \beta_k(\mcK(\sigma^-))$.
\item (Positive face) $\beta_{k-1}(\mcK(\sigma^-) \cup \{\sigma\}) = \beta_{k-1}(\mcK(\sigma^-))$ and $\beta_k(\mcK(\sigma^-) \cup \{\sigma\}) = \beta_k(\mcK(\sigma^-)) + 1$.
\end{enumerate}
Suppose that $w$ is an injective weight function. Then $w(\sigma)$ is a death time iff it is a negative face and $w(\sigma)$ is a birth time iff it is a positive face. 
\end{lemma}
 Positive faces may also  be characterized by $Z_k(\mcK(\sigma^-)) \subsetneq Z_k(\mcK(\sigma))$ or and negative faces by $B_k(\mcK(\sigma^-)) \subsetneq B_k(\mcK(\sigma))$. The first condition is equivalent to the statement that $\partial_k \sigma \in \partial_k C_k (\mcK(\sigma^-))$ and the second that $\partial_k \sigma \notin \partial_k C_k (\mcK(\sigma^-))$.

 With the classification of faces as positive and negative, we can assign labels to faces. We set $\ell(\sigma) = 1$, if $\sigma$ is positive and else $\ell(\sigma) = -1$ i.e., $\sigma$ is negative. The following theorem relating birth and death times to MSA was proven in \cite{Skraba17}.
\begin{theorem}(\cite[Theorem 3]{Skraba17})
\label{t:birthdeathmsa}
Let $\mcK$ be a finite weighted complex with an injective weight function $w$. Denote by $\dgm_b, \dgm_d$ respectively the birth and death times of the persistent diagram $\dgm(\mcK,w)$. Then we have that 
\begin{align*}
\dgm_d & = \{w(\sigma) : \ell(\sigma) = -1 \} \, = \{w(\sigma) : \sigma \in \MSA \} \nonumber \\
\dgm_b & = \{w(\sigma) : \ell(\sigma) = 1 \} \quad = \{w(\sigma) : \sigma \in \mcF_k \setminus \MSA \}.
\end{align*}
\end{theorem}

\subsection{Weak stabilization and central limit theorem for Poisson functionals}
\label{sec:weak_stab}

We now introduce Poisson processes formally and state the central limit theorem for statistics of Poisson processes. We refer the reader to \cite{last2017lectures,baccelli2020random} for more details on point processes and in particular, Poisson point processes.

Let $\cN$ be the set of all locally-finite point sets (or equivalently simple Radon counting measures) in $\mR^d$. By a locally-finite point set $\X$, we mean a point set $\X$ such that $\X(B) := | \X \cap B| < \infty$ for all bounded Borel subsets $B$. We shall also view $\X$ as a counting measure $\sum_{x \in \X} \delta_x$ and $\X$ is said to be a simple counting measure if $\X(\{x\}) \in \{0,1\}$ for all $x \in \mR^d$. The associated $\sigma$-algebra (called the {\em evaluation $\sigma$-algebra}) is the smallest $\sigma$-algebra such that $\X \mapsto \X(B)$ is measurable for all Borel subsets $B \subset \mR^d$. Here, we have used $|\cdot|$ to denote cardinality and later we shall also use it to denote Lebesgue measure.

By {\em a point process}, we mean a random element of the space $\cN$ i.e., a measurable mapping from a probability space $(\Omega,\mathcal{F},\mathbb{P})$ to $\cN$ equipped with the evaluation $\sigma$-algebra. A point process $\P$ is said to be a {\em stationary Poisson point process on $\mR^d$} with intensity $\lambda$ if it satisfies the following two properties:
\begin{enumerate}
\item For every Borel subset $B \subset \mR^d$, $\P(B)$ is a Poisson random variable with mean $\lambda |B|$.
\item For every $m \in \bN$ and pairwise disjoint Borel subsets $B_1,\ldots,B_m$ of $\mR^d$, the random variables $\P(B_1),\ldots,\P(B_m)$ are independent ({\em complete independence}).
\end{enumerate}
Given a bounded Borel subset $B$ of $\mR^d$, the restriction of Poisson point process to $B$ can be realized as follows:
$$\P \cap B \stackrel{d}{=} \{ X_1,\ldots,X_{N_B}\},$$
where $N_B$ is a Poisson random variable with mean $\lambda |B|$ and $X_1,X_2, \ldots$ are i.i.d. random vectors in $\mR^d$ with uniform distribution in $B$ and independent of $N_B$. We note that $\P$ is a locally-finite point set with points in general position.

We recall a general central limit theorem for functionals of Poisson point processes due to \cite{trinh2019central}. 
Recall that $W_n = \left[-\frac{n^{1/d}}{2},\frac{n^{1/d}}{2}\right]^d$, let $\mathfrak{A}$ ($=\mathfrak{A}(\{W_n\})$) be the collection of all subsets $A$ of $\mR^d$ such that $A = W_n + x$ for some $W_n$ in the sequence and some point $x \in \mR^d$. For any $\mathfrak{A}$-valued sequence $A_n, n \geq 1$, we say that $A_n \to \mR^d$,  if $\bigcup_{n \geq 1}\bigcap_{m \geq n} A_m \ =\ \mR^d$. 
\begin{theorem}(\cite[Theorem 3.1]{trinh2019central})
\label{thm:clt_Duy17}
Let $H$ be a real-valued functional defined for all finite subsets of $\mR^d$ and satisfying the following four conditions:
\begin{enumerate}[wide, labelindent=0pt]
\item[(i)] \emph{Translation invariance:} $H(\X + y) = H(\X)$ for all finite subsets $\X \subset \mR^d$ and $y \in \mR^d$.
\item[(ii)] \emph{Weak stabilization:} $H$ is said to be weakly stabilizing if there exists a random variable $D_{\infty}(H)$ such that for any $\mathfrak{A}$-valued sequence $A_n \to \mR^d$, the following holds.
\beq
\label{eqn:weakstab-rob}
(D_{\0}H)(\P \cap A_n) \stackrel{a.s.}{\rar} D_{\infty}(H), \, \mbox{as $n \to \infty$.}
\eeq
%
\item[(iii)] \emph{Poisson bounded moments:} $\sup_{A \in \mathfrak{A} ; \, 0 \in A} \EXP{\left[ (D_{\0}H)(\P \cap A)\right]^4} < \infty$.

\end{enumerate}
Then, there exists a constant $\sigma^2 \geq 0$, such that, as $n \to \infty$,
\[ n^{-1}\VAR{H(\P\cap W_n)} \to \sigma^2, \qquad n^{-1/2}\left[H(\P \cap W_n) - \EXP{H(\P\cap W_n)}\right]\ {\Rar}\ N(0,\sigma^2), \]
where  $\Rar$ denotes convergence in distribution and $N(0,\sigma^2)$ stands for normal random variable with mean $0$ and variance $\sigma^2$.\\  If $D_{\infty}(H)$ is a non-zero random variable (i.e., $\mathbb{P}(D_{\infty}(H) \neq 0) > 0$), then $\sigma^2 > 0$.
\end{theorem}
The last statement on variance lower bound follows from \cite[Remark 2.9]{trinh2022random}. The proof in \cite{Penrose01} uses a martingale-difference central limit theorem building upon the results of \cite{kesten1996central}. A newer proof via Poincar\'e inequality was given recently in \cite[Theorem 3.1]{trinh2019central} and compared to \cite{Penrose01}, this did not require a growth condition. We also note that the central limit theorem in \cite{Penrose01} is stated for a more general sequence of windows than $W_n$. 

\section{Proofs}
\label{s:proofs}

We first give a sketch of proof of stabilization for MST (i.e., in the special case of $k = 1$) in Section \ref{s:proof_MST} expanding upon the proof overview given in the introduction.  Next we recall some basic lemmas on weighted complexes in Section \ref{sec:basiclem} and then prove stabilization under chain maps in Section \ref{s:stabchainmaps}. The existence of chain maps between suitable Delaunay complexes is shown in Section \ref{s:lem_del}.  We prove stabilization of Poisson Delaunay complexes in Section \ref{s:stabPoisDel} and the proofs of main results are given in Section \ref{s:proofs_main} with the proof of variance lower bound alone postponed to Section \ref{s:varlb}. 

The above subdivision of sections is also to segregate the deterministic parts of the proof from the probabilistic parts. Sections \ref{sec:basiclem}-\ref{s:lem_del} are purely deterministic and the probabilistic proofs are in Sections \ref{s:stabPoisDel}-\ref{s:varlb}.

\subsection{Proof sketch in the case of \texorpdfstring{$k = 1$}.}
\label{s:proof_MST}

To better illustrate the proof strategy sketched at the end of Section \ref{s:intro}, we expound upon this now in the case of $k = 1$ and $\phi$ is identity. In this case $\pMi(\P_n)$ is simply the total weight of the minimal spanning tree formed on $\P_n$. Though accounting for boundary effects is an important part of our proof, we shall explain the proof steps for minimal spanning tree ignoring the boundary effects. It makes it easier to understand the crux of our proof.

Let $G := G(\X)$ be the Delaunay graph on a locally-finite point set $\X$  and recall that the edge-weights $w$ is given by the distance between the points. Consider also the Delaunay graph $G' := G(\X \cup \{\0\})$. Let $W_n = [-\frac{n^{1/d}}{2},\frac{n^{1/d}}{2}]^d, n \geq 0$ be the sequence of windows increasing to $\mR^d$ and $G_n, G'_n$ be the restriction of graphs $G,G'$ to $W_n$. Further, let $\MST_n : = \MST(G_n)$, $\MST'_n := \MST(G'_n)$ be the corresponding MSTs and $w(\MST_n), w(\MST'_n)$ be the total weight of the respective MSTs. Recall that for an edge $e$,  we denote the subgraph of $G_n$ consisting of edges whose weight is smaller than that of $e$ by $G_n(e^-)$.

\begin{enumerate}
\item Suppose that $G_n' = G_n - \{e_1,\ldots,e_k\} + \{e'_1,\ldots,e'_l\}$ for all $n \geq n_0$ and assume that the edges are all in $W_{n_0}$.
\item The add-one cost for MST upon adding a point to the point process can be decomposed as a finite sum of add-one cost (or its negative) of MST upon adding an edge i.e.,
\begin{align*}
& w(\MST'_n) - w(\MST_n) = \sum_{i=1}^k  w\big(\MST(G_n - \{e_1,\ldots,e_i\})\big) - w\big(\MST(G_n - \{e_1,\ldots,e_{i-1}\}) \big) \\
& + \sum_{i=1}^l  w\big(\MST(G_n - \{e_1,\ldots,e_k\} + \{e'_1,\ldots,e'_i\}) \big) - w\big( \MST(G_n - \{e_1,\ldots,e_k\} + \{e'_1,\ldots,e'_{i-1}\}) \big).
\end{align*}
\item Thus, we will show convergence of add-one cost for $\MST_n$'s of a suitable sequence of graphs under addition of a single edge and because of the above identity, this suffices to show convergence of $w(\MST'_n) - w(\MST_n)$. Hence we may assume that $G'_n = G_n \cup \{e_0\}$ for $n$ large enough, say $n \geq n_0$.

\item When an edge $e_0$ is added to the graph, the MST can increase at most by one and the symmetric difference of $\MST_n, \MST'_n$ is at most two for all $n$.
\item Suppose $e_0$ is positive for $G_n$ for some $n \geq n_0$, i.e., it creates a cycle in the graph for some $G_n(e_0^-)$, then it must also create a cycle for any larger graph $G_m(e_0^-)$ for $m > n$. In this case, the edge does not contribute to the difference of $\MST_m$, $\MST_m'$ for $m > n$ and so $\MST_m = \MST_m'$ for $m > n$.
\item If $e_0$ is negative for all $n \geq n_0$ i.e., it reduces the number of components when added for all $G_n(e_0^-)$, then the following three case can occur
\begin{enumerate}
\item If $\MST_m' = \MST_m + e_0$ for all $m \geq n_0$ (i.e., $e_0$ is not part of any cycle in $\MST_m$ for the entire sequence) then again stabilization follows easily.
\item If for some graph $G_n$, $\MST_n \setminus \MST'_n = e_n$ (for some edge $e_n \in G_n$) then for any larger graph $G_m$, $\MST_m \setminus \MST'_m = e_m$ where $w(e_m) \leq w(e_n)$. In other words, adding $e_0$ creates a cycle in $\MST_n$ such that there is a larger edge $e_n \in \MST_n$ and $\MST'_n = \MST_n - e_n + e_0.$ Then, the same holds for $\MST_m$ as well with an edge $e_m$ ($= e_n$ possibly) removed from $\MST_m$. Further, the edge $w(e_m)$ has to be at most $w(e_n)$ else $e_n$ itself would have been removed from $\MST_m$ as well.
\item Thus $w(\MST(G_n + e_0)) - w(\MST(G_n)) = w(e_0) - w(e_n)$ for all $n \geq n_0$ with $w(e_n)$ being a decreasing sequence and hence the add-one cost converges.
\end{enumerate}
\end{enumerate}
Formally,  we must consider minimal spanning trees on $G(\X_n)$ and $G(\X_n \cup \{\0\})$. The difference between $G(\X_n)$ and $G_n$ is that the latter has monotonicity in $n$ while the former does not. As discussed at the end of Section \ref{s:intro}, this is true for Delaunay complexes and we overcome this issue by using suitable chain maps between $G(\X_n)$ and $G(\X_m)$ for $m > n$.
\subsection{Basic lemmas on weighted complexes}
\label{sec:basiclem}

We shall fix a $k \geq 1$ in this subsection and assume all our minimal spanning acycles to be $k$-dimensional. Hence, for convenience, we shall drop the subscript $k$ in all our notations. We prepare for the main proofs with some preliminary lemmas on weighted complexes. Firstly, we recall the stability result from \cite[Theorem 4]{Skraba17} but stated for MSAs using \cite[Theorem 3]{Skraba17}. 
\begin{theorem}
\label{t:stabSTY17}(\cite[Theorem 4]{Skraba17})
Let $\mcK$ be a finite complex with two weight functions $w,w'$. Let $\MSA, \MSA'$ be the two induced minimal spanning acycles respectively. Then for any $p \in \{0,1, \ldots, \infty\}$ there exists a bijection between $\pi : \MSA \to \MSA'$ such that
$$ \sum_{\sigma \in \MSA}|w(\sigma) - w'(\pi(\sigma))|^p \leq \sum_{\sigma \in \mcF_k} |w(\sigma) - w'(\sigma)|^p,$$
where, $\sum_i |x_i|^0 := \sum_i \1[x_i \neq 0]$ and $\sum_i |x_i|^{\infty} = \sup_i |x_i|$ for a real sequence $x_i$.
\end{theorem}

The following lemma shall help us to understand how minimal spanning acycle changes upon addition of a simplex.
\begin{lemma}
\label{l:add_simplex}
Let $\mcK,\mcK_\sigma$ be weighted complexes on a finite point set $\X$ such that $\mcK_\sigma = \mcK \cup \{\sigma\}$ where $\sigma$ is a $k$-simplex such that $\sigma \notin \mcK$ and $\mcK_\sigma$ is a simplicial complex. We denote the labels of simplices on $\mcK$ by $\ell$ and that on $\mcK_\sigma$ by $\ell_\sigma$. Then, we have that

\begin{enumerate}[label=(\roman*),wide,  labelindent=0pt]
\item $| \MSA(\mcK_\sigma) \setminus \MSA(\mcK)|, | \MSA(\mcK) \setminus \MSA(\mcK_\sigma)| \leq 1$ or equivalently,
$$ \sum_{\tau \in \mcK} \1(\ell(\tau) \neq \ell_\sigma(\tau)) \leq 1.$$
\item With the notation as above, if $\ell_\sigma(\sigma) = 1$ then $\MSA(\mcK) = \MSA(\mcK_\sigma)$ and hence $\pM(\mcK) = \pM(\mcK_\sigma).$
\end{enumerate}
\end{lemma}
\begin{proof} \,
\noindent {\em (i):} This essentially follows from Theorem \ref{t:stabSTY17}. We shall use $w$ to denote the weight function on both $\mcK,\mcK_\sigma$. Define a weighted complex $\hat{\mcK}$ as follows : $\hat{\mcK} = \mcK \cup \{\sigma\}$ where $\hat{w}(\sigma) = \max_{\tau \in \mcK}w(\tau) + 1$. Observe that $\MSA(\mcK) \subset \MSA(\hat{\mcK}) \subset \MSA(\mcK) \cup \{\sigma\}$. Now, from Theorem \ref{t:stabSTY17}, we have that for any $p \geq 0$, there exists a bijection $\Pi : \MSA(\mcK_\sigma) \to \MSA(\hat{\mcK})$ such that
\[ \sum_{\tau \in \MSA(\mcK_\sigma)} |w(\tau) - \hat{w}(\Pi(\tau))|^p \leq |\hat{w}(\sigma) - w(\sigma)|^p.\]
For $p = 0$, the above inequality is nothing but $|\MSA(\mcK_\sigma) \setminus \MSA(\hat{\mcK})|,
|\MSA(\hat{\mcK}) \setminus \MSA(\mcK_\sigma)| \leq 1.$ Since $\MSA(\mcK) \subset \MSA(\hat{\mcK})$, first of the above inequalities yields immediately $|\MSA(\mcK) \setminus \MSA(\mcK_\sigma)| \leq 1$. If $\MSA(\mcK) = \MSA(\hat{\mcK})$, then the proof is complete.

Suppose $\MSA(\hat{\mcK}) = \MSA(\mcK) \cup \{\sigma\}$. Further, if $\sigma \notin \MSA(\mcK_\sigma)$ then $\sigma \notin \MSA(\hat{\mcK})$ as $w(\sigma) \leq \hat{w}(\sigma)$. Thus $\sigma \in \MSA(\mcK_\sigma)$ and so
$ \MSA(\mcK_\sigma) \setminus \MSA(\mcK) = \MSA(\hat{\mcK}) \setminus \MSA(\mcK) = \{\sigma\}$ and so the proof is complete. \\

\noindent {\em (ii):} Assume that $\MSA(\mcK) \neq \MSA(\mcK_\sigma)$. There are two possible cases: there exists a $\tau \in \MSA(\mcK) \backslash \MSA(\mcK_\sigma)$ or $\tau' \in \MSA(\mcK_\sigma) \backslash \MSA(\mcK)$. In the first case, this implies that the boundary of $\tau$ can be written as the linear combination using the boundary of $\sigma$. However since $\ell_\sigma(\sigma)=1$, this boundary can be expressed in terms of the boundaries of $\mcK$, contradicting $\tau\in \MSA(\mcK)$. In the latter case, we recall that for $\tau'\in \mcK\subseteq \mcK_\sigma$, $\ell_\sigma(\tau')=-1$ implies $\ell(\tau')=-1$ (\cite[Lemma 19]{Skraba17}) . This implies that $\tau'=\sigma$ which contradicts the assumptions. 

\end{proof}

\subsection{Stablization under Chain Maps}
\label{s:stabchainmaps}

In this section, we prove the topological lemmas regarding behaviour of weighted complexes under chain maps. Throughout this section we will consider the relationship of two weighted simplicial complexes $\mcK$ and $\mcK'$ (with the respective weight functions denoted $w,w'$) whose chain groups are related by a chain map, i.e.
$$ f: \bigoplus_k C_k(\mcK) \rightarrow \bigoplus_k C_k(\mcK') $$
We observe  that even if $\sigma \in \mcK \cap \mcK'$,  it need not hold that $w(\sigma) = w'(\sigma)$. 
We remind the reader that:  (1) $\partial$ is  the boundary map and that boundary of a simplex $\sigma$ is denoted by $\partial \sigma$; (2) since $\mathrm{dim}(\sigma) = k$, $\partial \sigma \in C_{k-1}$; (3) by definition, a chain map commutes with the boundary operator i.e.,  $f\circ \partial = \partial\circ f$.
\begin{assumption}\label{ass:chain_map}
We shall assume that the chain map satisfies following two properties:
\begin{enumerate}[wide, labelindent=0pt]
%
%
%
\item The restriction to any sublevel set of the weight function is a chain map i.e.,
$$ f: \bigoplus_k C_k(\mcK(t)) \rightarrow \bigoplus_k C_k(\mcK'(t)), $$
is a chain map for all $t \geq 0$.
\item For any $\sigma \in \mcK(t) \cap\mcK'(t)$ for some $t \geq 0$, $\sigma$ and $f(\sigma)$ are homologous as chains in $C_k(\mcK'(t))$.

\end{enumerate} 
\end{assumption}
Suppose that $\partial \sigma \in \mcK \cap \mcK'$ for a simplex $\sigma$. Then, we set
\begin{align*}
\mcK_\sigma &:= \mcK\cup\sigma\\
\mcK'_\sigma &:= \mcK'\cup\sigma
\end{align*}
For each of the complexes, we denote the label on the simplices as $\ell, \ell', \ell_\sigma$, and $\ell'_\sigma$ for $\mcK,\mcK', \mcK_\sigma$ and $ \mcK'_\sigma$ respectively. 

\newpage

\begin{lemma}\label{lem:neg_new}
\, 
\begin{enumerate}[wide,  labelindent=0pt]
\item[(i)] For any simplex $\tau\in \mcK\cap \mcK'$, if $\ell(\tau)=1$ then $\ell'(\tau)=1$.
\item[(ii)] If $\ell'_\sigma(\sigma)=-1$ then $\ell_\sigma(\sigma)=-1$.
\end{enumerate}
\end{lemma}

\begin{proof}
Note that it suffices to prove (i) as (ii) is its contrapositive applied to $\mcK_\sigma$ and $\mcK'_\sigma$. 
The proof of (i) follows from the existence of the chain map. 
\begin{equation}\label{eq:diag1}
\begin{tikzcd}
C_k(\mcK) \arrow["f_k"]{r}\arrow["\partial"]{d} & C_k(\mcK') \arrow["\partial'"]{d}\\
C_{k-1}(\mcK) \arrow["f_{k-1}"]{r} & C_{k-1}(\mcK')
\end{tikzcd}
\end{equation} 
By $\ell(\tau) = 1$, there exists a chain $c_k\in C_k(\tau^-)$ such that $\partial (\tau+c_k) = 0$. As $f$ is a chain map, by linearity, $f(\tau+c_k) = f(\tau) + f(c_k)$, and hence by the commutativity of $\partial$ and $f$,  we have that $\partial'(f(\tau) +f(c_k)) = 0$.
By assumption, $\tau$ and $f( \tau)$ are homologous as chains, so $\tau = f(\tau) + \partial' c_{k+1}$, where $c_{k+1}$ is some chain in $C_{k+1}(\mcK')$.
Applying the boundary operator, we obtain $\partial' (\tau +f(c_k) ) =  \partial' f(\tau) +\partial'f(c_k) +\partial'\partial' c_{k+1} = 0$, implying that $\ell'(\tau)=1$. 
\end{proof}
%
%
We will require one additional lemma which describes the relationship of the four complexes. %
\begin{lemma}
\label{lem:stab_negfac}
Let $\tau\in \mcK\cap\mcK'$ such that $\ell(\tau)=-1$ and $\ell_\sigma(\tau)=1$. If $\ell'_\sigma(\sigma) = -1$,
there exists a simplex $\tau'\in \mcK'$ such that 
\begin{itemize}
\item $\ell'(\tau')=-1$ 
\item $\ell'_\sigma(\tau')=1$ 
\item $w'(\tau')\leq w(\tau)$.
\end{itemize}
\end{lemma}
\begin{proof}
%
The assumptions imply that $\partial \sigma$ is a non-trivial cycle prior to the insertion of $\sigma$ in $C_k(\mcK)$, and so represents a non-trivial homology cycle in which must persist from at least $(a,w(\sigma))$ for some $a\leq w(\sigma)$. Note that this may represent an ephemeral class, i.e. $a=w(\sigma)$.  We observe that  $\ell'_\sigma(\sigma)=-1$ implies $\ell_\sigma(\sigma)=-1$ and $f(\sigma)=\sigma$. Note that this implies the weight of $\sigma$ remains unchanged. 

The assumptions further imply that $\partial \tau$ is homologous to $\partial \sigma$ in $\mcK(\tau^-)$. This can be seen since $\ell(\tau)=-1$, the chain $\partial\tau$ is a non-trivial homology class in $\mcK(\tau^-)$. However, since $\ell_\sigma(\tau)=1$, $\tau$ bounds the cycle $\partial_k\sigma$ in $\mcK$ and so the homology class $[\partial \sigma]$ is non-trivial until 
$w(\tau)$. 
Now we show the existence of $\tau'$. Consider the following commutative diagram 
\begin{equation}\label{e:commdiag1}
\begin{tikzcd}
\Hg_k (\mcK)\arrow["f"]{d} \arrow["\varphi"]{r} & \Hg_k (\mcK_\sigma)\arrow["f'"]{d}\\
\Hg_k (\mcK') \arrow["\varphi'"]{r} & \Hg_k (\mcK'_\sigma) 
\end{tikzcd}
\end{equation}
where $\varphi,\varphi',f,f'$ are the induced maps on homology groups. 
Consider the homology class $[\partial \sigma]$, i.e. the equivalence class of the cycle $\partial_k\sigma$. By assumption, as $\sigma$ is a negative simplex in all four complexes, it follows that $\partial\sigma$ is a representative of a non-trivial persistent homology class in an interval $(a,w(\sigma))$ for some $a\leq w(\sigma)$ in all four complexes. 
We remark that the bounds on the birth time, i.e. $a$, need not be the same in the four spaces, but are all upper-bounded by the value in $\Hg_k(\mcK)$.

By the properties of persistence module homomorphisms, the death time of an image (under $f$) of a persistence class is non-increasing. Since $[\partial \sigma]$ is trivial in $\Hg_k(\mcK')$ at $w(\tau)$, it follows that 
there must exist a bounding chain for $\psi([\partial \sigma])$ in $\Hg_k(\mcK')$ at or before $w(\tau)$. Note that this upper bound, is with respect to the filtration value in $\mcK$.
Set $\tau'$ to be the last simplex in this bounding chain (where the ordering is induced by $w'$). By construction, it is negative in $\mcK$ and by the insertion of $\sigma$ in $\mcK'$, it creates a new cycle -- as it bounds $\partial \sigma$, hence adding $\sigma$ creates a cycle, completing the proof. 
\end{proof}

\subsection{Existence of chain maps between Delaunay complexes}
\label{s:lem_del}

In order to use the results of the previous section, we need existence of chain maps between Delaunay complexes built on two finite point sets $\X \subseteq \Y$ such that they satisfy the assumptions in Assumption \ref{ass:chain_map}. We assume $\X,\Y$ are \emph{generic} in the sense that there are no co-linear or cocircular points. We first recall a few standard definitions and facts which are used in this section.
\begin{definition}
Given an open cover $\mathcal{U} := \{\mathcal{U}_i\}_{i\in \mathcal{I}}$ of a closed subset $ U \subset \mathbb{R}^d$, the nerve $\mathcal{N}(\mathcal{U})$ is the set of finite subsets of $\mathcal{I}$ defined as follows. A finite set $I\subseteq \mathcal{I}$ belongs to $\mathcal{N}(\mathcal{U})$ if and only if the intersection of the $\mathcal{U}_i$ whose indices are in $I$, is non-empty, i.e. 
$$\mathcal{U}_I = \bigcap\limits_{i\in I} \mathcal{U}_i \neq \emptyset$$
If $I$ belongs to $\mathcal{N}(\mathcal{U})$, then so do all of its subsets making $\mathcal{N}(\mathcal{U})$ an abstract simplicial complex. 
If a cover satisfies the condition that all finite non-empty intersections are contractible, then it is said to be a \emph{good cover}.
\end{definition}
%
If an open cover is good, the corresponding nerve captures the same topological information combinatorially. This is formalized in the Nerve Theorem; see \cite[Theorem 10.7]{bjorner1995topological} for example.
\begin{theorem} If $\mathcal{U}$ is a good open cover of a closed subset $U \subset \mathbb{R}^d$, then $U$ is homotopy equivalent to the nerve $\mathcal{N}(U)$.
\end{theorem}
We show that the sequence of simplex deletions and additions which occur when adding a point to $\X$ can be realized as a sequence of topological spaces with the Voronoi cells inducing a good cover of each space. This allows us to pass back and forth from the combinatorial description of the Delaunay complex to the geometric realization of the union of balls around point in $\X$ via the Nerve Theorem. This allows us to relate the Delaunay complexes on point sets $\X \subset \Y$, by defining a chain map which meets the requirements of Assumption \ref{ass:chain_map}. We remark more on the necessity for considering such a sequence of complexes at the end of this subsection.

First, we use the fact the Delaunay complex is the nerve of the cover induced by the Voronoi diagram of a generic point set $\mathcal{X}$. Define the $\varepsilon$-offset of the Voronoi cell of a point $x\in\mathcal{X}$ as 
$$\Vor^\varepsilon_\mathcal{X}(x) = \{p\in\mathbb{R}^d : \exists y \in \Vor_\mathcal{X}(x) \, \mbox{such that} \, |p-y|<\varepsilon \}$$ 
We now give a straightforward result whose proof will be sketched for completeness.
\begin{lemma}\label{lem:voronoi}
For $\varepsilon>0$ small enough, the cover induced by 
$$\Vor^\varepsilon_\mathcal{X}(x)\cap B_r(x) \qquad \forall x \in \X$$ 
form a good cover of the $\bigcup_{x\in \X} B_r(x)$ for all $r$. We call this cover {\em the Voronoi-induced cover} of $\X$.
\end{lemma}
\begin{proof}
As $\X$ is finite, we can choose $\varepsilon>0$ small enough such that the intersections of the $\varepsilon$-offsets of the Voronoi cells are in one-to-one correspondence with the adjacencies of the Voronoi cells.
The Voronoi cells (and their $\varepsilon$-offsets) are convex, hence they form a good cover of the full space. Additionally, the balls centered on the points are also convex, hence all intersections between balls and cells are either empty or convex, and hence contractible, implying the result. 
\end{proof}
Note that the choice of $\varepsilon$ depends on the point set and can be arbitrarily small. We simply require that for any finite generic point set (without multiplicity), an $\varepsilon>0$ exists. 
Assume that $\X \subset \Y$ are generic finite point sets as above. First, we define 
$$\X' = \X\cup\{\0\}, \qquad \qquad \Y' = \Y \cup\{\0\}, $$
and let $\mcK = \mcD(\X)$, $\mcL = \mcD(\Y)$ the Delaunay complex of $\X, \Y$ respectively. Similarly let $\mcK'$ and $\mcL'$ be the Delaunay complexes on $\X'$ and $\Y'$ respectively.  Observe that $\mcK,\mcK',\mcL,\mcL'$ are complexes that correspond to Voronoi-induced covers on $\X,\X',\Y,\Y'$ respectively. 

We assume that 
\begin{equation}
\label{e:KLdiffass}
\mcK \triangle \mcK' =\mcL \triangle \mcL' = \{\sigma_0,\sigma_1,\ldots,\sigma_k\}.
\end{equation}
We can choose $\varepsilon$ small enough, but dependent on the point sets $\X,\Y$, so that the nerve of this cover is combinatorially equivalent to the Delaunay complex. This follows from the finiteness of the point sets. 

There exists an ordering of the the simplices in \eqref{e:KLdiffass} such that 
$$\mcK^i = \begin{cases} \mcK^{i-1} - \sigma_i & 0 < i \leq l_1 \\ \mcK^{i-1} \cup \sigma_i & l_1 <i \leq l\end{cases} \quad ; \quad \mcL^i = \begin{cases} \mcL^{i-1} - \sigma_i & 0 < i \leq l_1 \\ \mcL^{i-1} \cup \sigma_i & l_1 <i \leq l \end{cases}$$ 
\begin{align}
\mcK & = \mcK^0 \supseteq \mcK^1 \supseteq \ldots \supseteq\mcK^{l_1} = \mcK \cap \mcK' \subseteq \mcK^{l_1+1} \subseteq \ldots \subseteq \mcK^{l} = \mcK', \nonumber \\
\label{e:insdelDnD'n} \mcL & = \mcL^0 \supseteq \mcL^1 \supseteq \ldots \supseteq\mcL^{l_1} = \mcL \cap \mcL' \subseteq \mcL^{l_1+1} \subseteq \ldots \subseteq \mcL^{l} = \mcL'.
\end{align}
In proving the main result of this section, we recount the following result
\begin{theorem}(\cite[Theorem 5.10]{bauer2017morse}) \label{thm:bauer}
 Let $\X$ be a finite set of points in general position in $\mathbb{R}^d$, then there is a simplicial collapse from the Delaunay-\v Cech complex to the Delauany complex for any weight $w\in \mathbb{R}$. 
\end{theorem}
\begin{lemma}
\label{lem:chainKnm}
For each $i$, there exists a chain map 
$$f_i:\bigoplus\limits_k C_k(\mcK^{i}) \rightarrow \bigoplus\limits_k C_k(\mcL^{i}),$$
which satisfies the conditions listed in Assumption \ref{ass:chain_map}.
\end{lemma}
\begin{proof}
 Define
\begin{align*}
X_r &= \bigcup\limits_{p \in \X} B_r(p)\cap \Vor_\X(p)\\
Y_r &= \bigcup\limits_{q \in \Y} B_r(p)\cap \Vor_\Y(q) 
\end{align*}
By standard results for the Delaunay complex, there exist chain maps for $i=0,L$ i.e., 
\begin{equation}
\label{e:chainmapsKnm}
C_*(\mcK_r) \xhookrightarrow{\alpha} C_*(X_r) \xhookrightarrow{\beta} C_*(Y_r) \xrightarrow{\gamma} C_*(\mcL_r).
\end{equation}
Note that from this point on, we omit the radius/weight parameter as $\alpha$ and $\gamma$ are homotopy equivalences for any value of the radius and they commute with inclusion, e.g. ~\cite{bauer2017morse}.
For $i\neq 0,L'$, the homotopy equivalences are no longer obvious as the underlying complex is not simply the nerve of a cover of the space, as we are deleting and adding simplices. To prove the existence of chain maps, we introduce a sequence of topological spaces where the nerve of the Voronoi-induced cover is $\mcK^i$ and forms a good cover, implying the existence of the homotopy equivalences and the chain maps. 
We define the topological spaces corresponding to $X^i(r)$ and $Y^i(r)$.

For each simplex $\tau$ in a Delaunay complex, there is a corresponding face in the Voronoi diagram. Denote this face by $\nu^\X(\tau)$ and its $\varepsilon$-offset by $\nu^\X_\varepsilon(\tau)$. 
\begin{align*}
X^i(r)&= \bigcup\limits_{p \in \X} B_r(p)\cap \Vor_{\P}(p)-\bigcup\limits_{\tau \in \sigma_i} \nu^\X_\varepsilon(\tau)\\
Y^i(r)&= \bigcup\limits_{q \in \Y} B_r(q)\cap \Vor_\Y(q)-\bigcup\limits_{\tau \in \sigma_i} \nu^\Y_\varepsilon(\tau)
\end{align*} 
where $\varepsilon$ is arbitrarily small and depends on $\X$ and $\Y$. As $\nu^\Y(\tau) \subseteq \nu^\X(\tau)$, it follows that $X^i(r)\subseteq Y^i(r)$. 

By construction the nerves of these covers are $\mcK^i$ and $\mcL^i$ respectively, as the intersections corresponding to missing simplices (and their cofaces) are empty, since they have been removed from the space. It remains a good cover, as all the cover elements and non-empty intersections are star-shaped and so contractible. We can now define the chain map 
$$\psi_i = \gamma_i \circ \beta_i \circ \alpha_i$$
Applying the standard good cover homotopy equivalence argument to these intermediate spaces we obtain a chain map. This proves the existence of a chain map satisfying Assumption \ref{ass:chain_map}(1).

To show that the map satisfies Assumption~\ref{ass:chain_map} (2), we show that the following diagram for $s<t$,
\begin{center}
\begin{tikzcd}
\mcK^i_s \arrow{r}{f_s}\arrow{d} & \mcL^i_{s}\arrow{d}\\
\mcK^i_t \arrow{r}{f_t} & \mcL^i_{t}
\end{tikzcd}
\end{center}
commutes up to chain homotopy, which implies the that $f(\sigma)$ and $\sigma$ are homologous. If $\sigma\in \mcK^i_s \cap \mcL^i_s$, then $f_s(\sigma)=f_t(\sigma)=\sigma$, and the diagram commutes. 
However, we must also show that the diagram commutes when $\sigma\in \mcL^i_t\backslash \mcL^i_s$. In this case, we insert an intermediate step in the above diagram using the Delaunay-\v Cech complex, which we denote by $\mcD\mcC$. We observe that it can be constructed as a cover of the $i$-th space, so we denote the corresponding complex by $\mcD\mcC^i$. This follows from the fact that  $\mcD\mcC$ is equal to the Delaunay complex, with a different weight function (see  Definition \ref{def:delcomp}). Including the complex in the diagram, we obtain
\begin{center}
\begin{tikzcd}
\mcK^i_s \arrow{r}\arrow{d} & \mcD\mcC^i_s \arrow{r}\arrow{d} &\mcL^i_{s}\arrow{d}\\
\mcK^i_t \arrow{r} & \mcD\mcC^i_t\arrow{r}&\mcL^i_{t}
\end{tikzcd}
\end{center}
As the weight function for the Delaunay-\v Cech cannot decrease all the maps in the left-hand square are inclusions, so the square commutes up to homotopy. By Theorem~\ref{thm:bauer}, the horizontal maps  $\mcD\mcC^i \rightarrow \mcL^i$ are simplicial collapses which commute with inclusions, hence the original diagram commutes up to homotopy.


\end{proof}

\begin{remark}
At first glance, the intermediate step of the spaces may seem unnecessary as one could pass to the \v Cech complex rather than the topological spaces. However, in this case inserting and removing simplices is complicated by the fact that we need to insert and remove the corresponding cofaces in the \v Cech complex, which no longer corresponds to the cover by balls of a given radius complicating the analysis. By constructing the sequence of spaces, we are able to define the appropriate maps for all intermediate steps which allows us to use the results of the previous section.
\end{remark}

\subsection{Stabilization of Poisson Delaunay Complexes}
\label{s:stabPoisDel}

To use the approach of Section \ref{s:stabchainmaps} for showing stabilization of death times in Poisson-Delaunay complex, we need to establish \eqref{e:KLdiffass} which is essentially the  stabilization of simplices in the Poisson-Delaunay complex. Recall the set of boxes $\mathfrak{A}$ defined before Theorem \ref{thm:clt_Duy17}.
\begin{prop}
\label{prop:const_diff}
Let $\P$ be a stationary Poisson point process of unit intensity and $A_n \in \mathfrak{A}$ be a sequence of sets such that $A_n \to \mR^d$. Let $\mcD_n = \mcD(\P \cap A_n)$ and $\mrD_n = \mcD((\P \cap A_n) \cup \{\0\})$ be the sequence of Delaunay complexes on $\P \cap A_n$ and  $(\P \cap A_n) \cup \{\0\}$, $n \geq 1$ respectively. There exist a.s. finite random variables $L,M$ such that for all $n > M$,
$$ D_n \bigtriangleup \mrD_n = \{\sigma_1,\ldots,\sigma_L\}$$
and furthermore $w(\sigma_i) \leq M$ for all $i \leq L$. 
\end{prop}
The proof is a modification of arguments from \cite[Lemma 5.1]{Penrose2007laws} with more quantitative estimates provided in Proposition \ref{prop:bddmomentsA}. Similar arguments are used in the case of $d = 2$ in \cite[Section 4]{McGivney99} and \cite[Section 8]{Penrose01}.
\begin{proof}
Let $\cone_i, 1 \leq i \leq m$ be infinite open cones with apex at $\0$ and angular radius $\pi/6$ such that $\mR^d = \cup_i \cone_i$. Let $R_i = d(\0, \P \cap \cone_i)$ i.e., distance from origin to the closest point in the cone $\cone_i$. Define $R = \max R_i$. Since $\cone_i$'s have infinite volume, $\P \cap \cone_i \neq \emptyset$ a.s. and so $R_i < \infty$ a.s. for all $1 \leq i \leq m$. Thus $R < \infty$ a.s. and also adding $\0$ does not affect the Voronoi cells outside $B_{R}(\0)$ i.e., if $\Vor_{\P_n \cup \{\0\}}(X) \cap \Vor_{\P_n \cup \{\0\}}(\0) \neq \emptyset$ or $\0 \in \Vor_{\P_n}(X)$ then $|X| \leq R$. This also gives that $\Vor_{\P_n \cup \{\0\}}(X) = \Vor_{\P_n}(X)$ if $|X| \geq 2R$. So, we have that if $\sigma \in \mcD(\P \cap W_j) \bigtriangleup  \mrD(\P \cap W_j)$ then $\sigma \cap B_{2R}(\0) \neq \emptyset$. This implies that Delaunay faces outside $B_{3R}(\0)$ are unchanged by adding $\0$. 

We now consider the case $A_n = W_n, n \geq 1$. Set $M' := 3R$. Thus in this case for all $j \geq M'$
$$ \mcD(\P \cap W_j) \bigtriangleup \mrD(\P \cap W_j) \subset \{ \sigma \in D_j : \sigma \subset B_{3R}(\0) \} \cup \{ \sigma' \in \mrD_j : \sigma' \subset B_{3R}(\0) \}.$$
If $\sigma \subset B_{3R}(\0)$, then $w(\sigma) \leq 3R$ and similarly for $\sigma'$. So, we get the first claim in the Proposition when $A_n = W_n$, $n \geq 1$ and with $M = M'$. As for $L$ in this case, observe that 
$$|\mcD(\P \cap W_M)| \leq |\P \cap W_M|^{k+1} \, \, ; |\mrD(\P \cap W_M)| \leq (|\P \cap W_M| +1)^{k+1},$$
and these bounds together with a.s. finiteness of $M$ give the a.s. finiteness of $L$. 

Set $M' = 3R$. Suppose that $A_n \to \mR^d$ is an arbitrary sequence of boxes in $\mathfrak{A}$. Then there exists an a.s. finite random variable $M$ such that $\P \cap W_{M'} \subset \P \cap A_M$ as $\P \cap W_{M'}$ is a.s. finite. Thus by the above argument, we obtain that for $j > M$ 
$$ \mcD(\P \cap A_j) \bigtriangleup \mrD(\P \cap A_j) \subset \{ \sigma \in D_j : \sigma \subset B_{3R}(\0) \} \cup \{ \sigma' \in \mrD_j : \sigma' \subset B_{3R}(\0) \}.$$
The a.s. finiteness of $L$ in this case follows as in the special case of $A_n = W_n$. 
\end{proof}

As indicated after Proposition \ref{prop:const_diff}, we provide some tail estimates for the random variables therein and for arbitrary boxes $A \in \mathfrak{A}$ to verify the moment conditions for our CLT.
\begin{prop}(\cite[Section 5.1]{Penrose2007laws})
\label{prop:bddmomentsA}
Let $\P$ be a stationary Poisson point process of unit intensity and let $A \in \mathfrak{A}$ such that $\0 \in A$. Let $\mcD_A = \mcD(\P \cap A)$ and $\mrD_A = \mcD((\P \cap A) \cup \{\0\})$ be the Delaunay complexes on $\P \cap A$ and $(\P \cap A) \cup \{\0\}$ respectively. There exists a.s. finite random variables $L_A,M_A$ such that
$$ D_A \bigtriangleup \mrD_A = \{\sigma_1,\ldots,\sigma_{L_A}\}$$
and further $w(\sigma_i) \leq M_A$ for all $i \leq L_A$. Also, we have that $L_A, M_A$ have exponentially decaying tails uniformly in $A$ i.e., there exist constants $C,c > 0$ (depending on $d$) such that for all $t \geq 1$, 
$$ \sup_{A \in \mathfrak{A}, \0 \in A} \BP(M_A > t) \leq Ce^{-ct^d} \, \, ; \, \,  \sup_{A \in \mathfrak{A}, \0 \in A} \BP(L_A > t) \leq Ce^{-ct^{1/(d+1)}} .$$
\end{prop}
We sketch the proof from \cite[Lemma 5.1]{Penrose2007laws} for sake of completeness. 
\begin{proof}
Let $\cone_i, 1 \leq i \leq m$ be infinite open cones with apex at $\0$ and angular radius $\pi/12$ such that $\mR^d = \cup_{i=1}^m \cone_i \cup \{\0\}$. Let $\cone_i^+$ be the open cone concentric to $\cone_i$ with apex at $\0$ and angular radius $\pi/6$. Set $d_i(A) := \diam(\cone_i \cap A)$. Define
\begin{align*}
R_i(A) & := \min \{ \min \{ |X| : X \in \P \cap \cone^+_i \cap B_{d_i(A)}(\0) \}, d_i(A) \}, 1 \leq i \leq m, \\
R(A) & := \max_{1 \leq i \leq m} R_i(A). 
\end{align*}
By geometric considerations, we see that $\Vor_{(\P \cap A) \cup \{\0\}}(\0) \cap \cone_i \subset B_{R_i(A)}(\0) \cap \cone_i$ and so $\Vor_{(\P \cap A) \cup \{\0\}}(\0) \subset B_{R(A)}(\0)$. Thus, arguing as in the proof of Proposition \ref{prop:const_diff}, we have that
\[ D_A \bigtriangleup \mrD_A \subset B_{3R_A}(\0). \]
Now, we shall derive tail bounds for $R_A$ and using that, derive suitable bounds for $M_A, L_A$. Suppose that for some $1 \leq i \leq m$, $R_i(A) \geq t$ for some $t \geq 1$. Then there is a $y \in \cone_i \cap A$ such that $|y| = t$ and also $B_{at}(y/2) \subset \cone^+_i$ for some $a > 0$. Here $a$ depends on $d$ but not on $A$. Since $R_i(A) \geq t$, it holds that $\P \cap A \cap B_{at}(y/2) = \emptyset$. As $A = W_n+x$ for some $n,x$ and $\0 \in A$, we have that $y/2 \in A$ and $|A \cap B_{at}(y/2)| \geq c'a^dt^d$ for some constant $c' > 0$ dependent on $d$ but not on $A$. Thus, we derive that for $t \geq 1$, 
$$ \BP(R_i(A) \geq t) \leq \BP(\P \cap A \cap B_{at}(y/2) = \emptyset) \leq e^{-c'a^dt^d}.$$
This yields that if $\sigma \in D_A \bigtriangleup \mrD_A$ then $w(\sigma) \leq M_A := 3R_A$ and the required bound as well. 

To bound $L_A$ observe that $L_A \leq 2(d+1) (|\P \cap B_{M_A}(\0)| + 1)^{d+1}$ as $|\P \cap B_{M_A}(\0)| + 1$ is the total number of vertices in $(\P \cap A \cap B_{M_A}(\0)) \cup \{\0\}$ and $d$ is the maximum dimension of the Delaunay complex. For $t \geq 1$, set $s = s(t)$ such that $\theta_ds^d= (\frac{t}{2(d+1)})^{1/(d+1)}-1)$ where $\theta_d$ is the volume of the unit ball. Using the concentration inequality for Poisson random variables from \cite[Lemma 1.4]{Penrose03}, we can derive that
\begin{align*}
\BP(L_A \geq t) & \leq \BP\Big(|\P \cap B_{s/2}(\0)| \geq \theta_ds^d \Big) + \BP(M_A \geq s/2) \\
& \leq \BP\Big( \Big| |\P \cap B_{s/2}(\0)| - \theta_d(s/2)^d | \geq \theta_d(1-1/2^d)s^d \Big) + Ce^{-cs^d/2^d} \\
& \leq \tilde{C}e^{-\tilde{c}s^d} + Ce^{-cs^d/2^d}.
\end{align*}
Now substituting the choice of $s$ into the above bound completes the proof of the proposition. 
\end{proof}

\subsection{Proofs of main results - Proposition \ref{prop:stab_death_times} and Theorem \ref{thm:clt_msa}.}
\label{s:proofs_main}

We now prove weak stabilization of the sum of death times in the Poisson-Delaunay complex and thereby prove also weak stabilization of the sum of birth times and lifetimes as well. We use this to prove the central limit theorem for all three quantities. The starting point is a stablization result for the sum of death times for Delaunay complexes on locally finite point sets but along increasing windows $W_n$. 
\begin{prop}
\label{prop:stabilization_D}
Let $\X$ be a locally-finite point set with points in general position and  as before, $W_n = [-\frac{n^{1/d}}{2},\frac{n^{1/d}}{2}]^d$, $n \geq 1$. Set $\X_n = \X \cap W_n$ and let $\mcD_n = \mcD(\X_n), \mrD_n = \mcD(\X_n \cup \{0\})$ and $\mcD = \mcD(\X)$ be the Delaunay complexes on $\X_n, \X_n \cup \{\0\}$ and $\X$ respectively. Assume that $\mcD_n$ stabilizes i.e., there exists a $n_0$ such that for all $n \geq n_0$, 
\begin{equation}
\label{e:stabD}
\mcD_n \triangle \mrD_n = \{\sigma_1,\ldots,\sigma_l\},
\end{equation}
where $\sigma_k \subset W_{n_o}$ for all $1 \leq j \leq l$. Then there exists a constant $\pM_k(\mcD) \in [0,\infty)$ such that $\pM_k(\mrD_n) - \pM_k(\mcD_n) \stackrel{n \to \infty}{\to} \pM_k(\mcD).$
\end{prop}
\begin{proof}
We shall fix $k \in \{1,\ldots,d\}$ and drop the subscript $k$ everywhere. Let $n \geq n_0$. Firstly, let us set that $\{\sigma_1,\ldots,\sigma_{l_1}\} = \mrD_n \setminus ( \mrD_n \cap \mcD_n)$ and $\{\sigma_{l_1+1},\ldots,\sigma_{l}\} = \mcD_n \setminus ( \mrD_n \cap \mcD_n)$. Thus, as in \eqref{e:insdelDnD'n}, we have the following decomposition 
$$ \mcD_n = \mcK^0_n \supseteq \mcK^1_n \supseteq \ldots \supseteq\mcK^{l_1}_n = \mcD_n \cap \mrD_n \subseteq \mcK^{l_1+1}_n \subseteq \ldots \subseteq \mcK^{l}_n = \mrD_n,$$
where $\mcK^{j}_n \triangle \mcK^{j-1}_n = \{\sigma_j\}$. Thus, we obtain that
\begin{equation*}
  \pM(\mcD_n) - \pM( \mrD_n) = \sum_{j=1}^{l} \pM(\mcK^{j-1}_n) - \pM(\mcK^j_n)
\end{equation*}
and so convergence of $ \pM(\mcD_n) - \pM( \mrD_n))$ follows if we show that for all $1 \leq j \leq l$ that
\begin{equation}
\label{e:cvgdnj}
\lim_{n \to \infty} \pM(\mcK^j_n) - \pM(\mcK^{j-1}_n) = \pM(\mcK^j).
\end{equation}
The rest of the proof will be devoted to showing \eqref{e:cvgdnj}. Without loss of generality, we assume $\mcK^j_n = \mcK^{j-1}_n \cup \{\sigma_j\}$. If $\sigma_j$ is not a $k$-face, $\pM(\mcK^j_n) - \pM(\mcK^{j-1}_n) = 0$ for all $n$ and so we assume that $\sigma_j$ is a $k$-face. Denote the label of a simplex in $\mcK^j_n$ by $\ell^j_n$ i.e., $\ell^j_n(\sigma) := \ell_{\mcK^j_n}(\sigma)$ for $\sigma \in \mcK^j_n$. Since $\X_n \subset \X_m$ for $m \geq n$, Lemma \ref{lem:chainKnm} gives us a chain map from $\mcK^j_n$ to $\mcK^j_m$ satisfying assumptions in Assumption \ref{ass:chain_map} for all $m \geq n$. Thus, we shall use the results on stabilization under chain maps in Section \ref{s:stabchainmaps}. 
We shall now prove \eqref{e:cvgdnj} by breaking into 3 cases. 
\begin{itemize}[labelwidth=*,align=left]
\item[Case (i):] Suppose that $\ell^j_m(\sigma_j) = 1$ for some $n \geq n_0$. Thus, by Lemma \ref{lem:neg_new}(i) $\ell^j_m(\sigma_j) = 1$ for all $m \geq n$ and so we have by Lemma \ref{l:add_simplex}(2) that $\MSA(\mcK^j_m) = \MSA(\mcK^{j-1}_m)$ for all $m \geq n$. Consequently $\pM(\mcK^j_m) = \pM(\mcK^{j-1}_m)$ for all $m \geq n$. Thus \eqref{e:cvgdnj} holds with limit being $0$. 

\item[Case (ii):] Suppose that $\ell^j_n(\sigma_j) = -1$ for all $n \geq n_0$. Thus $\MSA(\mcK^j_n) \setminus \MSA(\mcK^{j-1}_n) = \{ \sigma_j \}$ (see Theorem \ref{t:birthdeathmsa} and Lemma \ref{l:add_simplex}(1)). We subdivide this into two further cases. \\
\begin{itemize}[wide,  labelindent=0pt]
\item[(a):] Suppose $\MSA(\mcK^{j-1}_n) \setminus \MSA(\mcK^j_n) = \emptyset$ for all $n \geq n_0$ i.e., no negative faces in $\mcK^{j-1}_n$ turns positive upon adding $\sigma_j$. Then we have that $\pM(\mcK^j_n) - \pM(\mcK^{j-1}_n) = \phi(w(\sigma_j))$ and so \eqref{e:cvgdnj} holds trivially with the limit being $\phi(w(\sigma_j))$. \\

\item[(b):] Suppose that $\MSA(\mcK^{j-1}_n) \setminus \MSA(\mcK^j_n) \neq \emptyset$ for some $n \geq i_0$. Let us set $\{\tau^j_n\} = \MSA(\mcK^{j-1}_n) \setminus \MSA(\mcK^j_n)$ i.e., $\ell_n^j(\tau^j_n) = 1$ and $\ell_n^{j-1}(\tau^j_n) = -1$. 

From Lemma \ref{lem:stab_negfac}, we have that for all $m \geq n$, there exists $\tau^j_m \in \mcK^{j-1}_m$ such that $\ell_m^j(\tau^j_m) = 1$ and $\ell_m^{j-1}(\tau^j_m) = -1$. Furthermore, $w(\tau^j_m)$ is decreasing in $m$. Thus $w(\tau^j_n)$ converges as $n \to \infty$ and consequently
$$ \lim_{n \to \infty} \pM(\mcK^j_n) - \pM(\mcK^{j-1}_n) = \phi(w(\sigma_j)) - \lim_{n \to \infty}\phi(w(\tau^j_n)) =: \pM(\mcK^j).$$
Hence \eqref{e:cvgdnj} holds with the limit being $\pM(\mcK^j)$ as defined above. 
\end{itemize}
\end{itemize}
\end{proof}

\begin{proof}[Proof of Proposition \ref{prop:stab_death_times}.]
We shall again fix $k \in \{1,\ldots,d\}$ and drop the subscript $k$ everywhere. Suppose $\pF(\mcD(\P \cap A_n))$ and $\pM(\mcD(\P \cap A_n))$ satisfy weak stabilization (i.e., \eqref{e:WstabBn} holds) then so does $\pB(\mcD(\P \cap A_n))$. This is because of linearity of weak stabilization and that $\pB(\mcD(\P \cap A_n)) = \pF(\mcD(\P \cap A_n)) - \pM(\mcD(\P \cap A_n))$. Thus, to conclude the proof, we show weak stabilization of $\pF(\mcD(\P \cap A_n))$ and $\pM(\mcD(\P \cap A_n))$. Weak stabilization of $\pF(\mcD(\P \cap A_n))$ follows from Proposition \ref{prop:const_diff} as $D_{\0}(F^{\phi}(\mcD_n))$ remains unchanged for $n > M$ and so the rest of the proof will be devoted to showing weak stabilization of $\pM(\mcD(\P \cap A_n))$.

Assume that we have an arbitrary sequence of sets $A_n$ as in the statement of the proposition i.e., $A_n \to \mR^d, A_n \in \mathfrak{A}$. Hence, we have that for all $n \geq 1$, there exists an a.s. finite random variable $M_n$ and also a finite $m_n \in \bN$ such that the following inclusions hold.
\begin{equation}
\label{e:inclusionBnWn}
A_n \subset W_{m_n} \, \, \mbox{and} \, \, \P \cap W_n \subset \P \cap A_m, \, \forall m \geq M_n.
\end{equation}
We can assume without loss of generality that $m_n \geq n$ and $M_n \geq n$ a.s. 

By Proposition \ref{prop:const_diff}, there exists a.s. finite random variables $N_0,L$ such that for all $n \geq N_0$, we have that a.s., 
\begin{equation}
\label{e:stabDWnDBn}
\mcD(\P \cap W_n) \triangle \mrD(\P \cap W_n) = \mcD(\P \cap A_n) \triangle\mrD(\P \cap A_n) = \{\sigma_1,\ldots,\sigma_L\} \subset W_{N_0}. 
\end{equation}
Furthermore, as in \eqref{e:insdelDnD'n}, we have the following decompositions 
$$ \mcD(\P \cap W_n) = \mcK^0_n \supseteq \mcK^1_n \supseteq \ldots \supseteq\mcK^{L_1}_n = \mcD_n \cap \mrD_n \subseteq \mcK^{L_1+1}_n \subseteq \ldots \subseteq \mcK^{l}_n = \mrD(\P \cap W_n),$$
and
$$ \mcD(\P \cap A_n) = \mcL^0_n \supseteq \mcL^1_n \supseteq \ldots \supseteq\mcL^{L_1}_n = \mcD_n \cap \mrD_n \subseteq \mcL^{L_1+1}_n \subseteq \ldots \subseteq \mcL^{l}_n = \mrD(\P \cap A_n),$$
where $\mcK^{j}_n \triangle \mcK^{j-1}_n = \mcL^{j}_n \triangle \mcL^{j-1}_n = \{\sigma_j\}$. 
So we obtain that
$$ D_{\0}( \pM(\mcD(\P \cap A_n)) = \pM(\mcD(\P \cap A_n)) - \pM( \mrD(\P \cap A_n)) = \sum_{j=1}^{L} \pM(\mcL^{j-1}_n) - \pM(\mcL^j_n)$$
From Propositions \ref{prop:stabilization_D} and \ref{prop:const_diff},  we have that
\begin{equation}
\label{e:cvgDKnj}
\lim_{n \to \infty} \pM(\mcK^j_n) - \pM(\mcK^{j-1}_n) = \pM(\mcK^j),
\end{equation}
where the limit $\pM(\mcK^j)$ is explicitly identified in Cases (i), (ii)(a) and (ii)(b) of the proof of Proposition \ref{prop:stabilization_D}.
So the a.s. convergence of $D_{\0}( \pM(\mcD(\P \cap A_n))$ follows if we show that for all $1 \leq j \leq L$ that a.s., 
\begin{equation}
\label{e:cvgDBnj}
\lim_{n \to \infty} \pM(\mcL^j_n) - \pM(\mcL^{j-1}_n) = \pM(\mcK^j)
\end{equation}
with $\mcK^j_n$ as in \eqref{e:cvgDKnj}.

The labels $\ell_n$ are with respect to $A_n$ but to distinguish labels with respect to $W_n$ , the latter are denoted by $\ell_{W_n}$. We also use $\ell_n^j, \ell_n^{j-1}$ to represent labels with respect to $\mcL^j_n, \mcL^{j-1}_n$ respectively. We shall work with realizations for which $N_0,L < \infty$ and furthermore \eqref{e:stabDWnDBn} holds. We shall again divide into cases as in the proof of Proposition \ref{prop:stabilization_D}. 

Without loss of generality assume that $\mcK_n^j = \mcK_n^{j-1} \cup \{\sigma_j\}$ and $\mcL_n^j = \mcL_n^{j-1} \cup \{\sigma_j\}$.
\begin{itemize}[labelwidth=*,align=left]
\item[Case (i):] Suppose that $\ell^j_{W_n}(\sigma_j) = 1$ for some $n \geq N_0$. From \eqref{e:inclusionBnWn}, \eqref{e:stabDWnDBn} and Lemma \ref{lem:chainKnm}, we obtain that there exists a chain map from $\mcK^j_n$ to $\mcL^j_m$ satisfying Assumption \ref{ass:chain_map} for all $m \geq M_n$. Thus $\ell^j_m(\sigma_j) = 1$ for all $m \geq M_n$. So we have that $\MSA(\mcK^j_m) = \MSA(\mcK^{j-1}_m)$ for all $m \geq M_n$ and consequently $\pM(\mcK^j_m) = \pM(\mcK^{j-1}_m)$ for all $m \geq M_n$ (see Lemmas \ref{lem:neg_new}(i) and \ref{l:add_simplex}(ii)). Hence \eqref{e:cvgDBnj} holds with limit being $0$ for both the LHS and RHS. 

\item[Case (ii):] Suppose that $\ell^j_{W_n}(\sigma_j) = -1$ for all $n \geq N_0$. Using Lemma \ref{lem:chainKnm}, \eqref{e:inclusionBnWn} and \eqref{e:stabDWnDBn}, we see that there is chain map from $\mcL^j_n$ and $\mcK^j_{m_n}$ for all $n \geq N_0$ that satisfies Assumption \ref{ass:chain_map}. Thus we have that $\ell^j_{n}(\sigma_j) = -1$ for all $n \geq M_0 := M_{N_0}$ (see Lemma \ref{lem:neg_new}(ii)). So by the stability result (Lemma \ref{l:add_simplex}(i)) ,we have that $\MSA(\mcK^j_n) \setminus \MSA(\mcK^{j-1}_n) = \MSA(\mcL^j_n) \setminus \MSA(\mcL^{j-1}_n) = \{ \sigma_j \}$ for all $n \geq M_0$. We subdivide this into two further cases. \\ 

\begin{itemize}[wide, labelindent=0pt]
\item[(a):] Suppose $\MSA(\mcL^{j-1}_n) \setminus \MSA(\mcL^j_n) = \emptyset$ for all $n \geq M_0$ i.e., no negative faces in $\mcL^{j-1}_n$ turns positive upon adding $\sigma_j$. Then we have that $\pM(\mcL^j_n) - \pM(\mcL^{j-1}_n) = \phi(w(\sigma_j))$ and so \eqref{e:cvgDBnj} holds trivially with the limit being $\phi(w(\sigma_j))$ for both the LHS and RHS. \\

\item[(b):] Suppose that $\MSA(\mcL^{j-1}_n) \setminus \MSA(\mcL^j_n) \neq \emptyset$ for some $n \geq M_0$. Let us set $\{\tau^j_n\} = \MSA(\mcL^{j-1}_n) \setminus \MSA(\mcL^j_n)$ i.e., $\ell_m^j(\tau^j_n) = 1$ and $\ell_n^{j-1}(\tau^j_n) = -1$. 

From \eqref{e:inclusionBnWn}, \eqref{e:stabDWnDBn} and Lemma \ref{lem:chainKnm}, we have that there exists a chain map between $\mcL^j_n$ and $\mcK^j_{m_n}$. Lemma \ref{lem:stab_negfac} now gives that for all $m \geq m_n$, there exists $\pi^j_m \in \mcK^{j-1}_m$ such that $\ell_{W_m}^j(\pi^j_m) = 1$, $\ell_{W_m}^{j-1}(\pi^j_m) = -1$ and $w(\pi^j_m) \leq w(\tau^j_n)$. Thus using that $\phi$ is increasing 
$$ \liminf_{n \to \infty} \phi(w(\tau^j_n)) \geq \liminf_{m \to \infty}\phi(w(\pi^j_n)) = \phi(w(\sigma_j)) -\pM(\mcK^j),$$
where the limit follows from Case(ii)(b) in the proof of Proposition \ref{prop:stabilization_D}. Now again using \eqref{e:inclusionBnWn}, \eqref{e:stabDWnDBn}, Lemmas \ref{lem:chainKnm} and \ref{lem:stab_negfac}, we can deduce that $w(\pi^j_n) \geq w(\tau^j_m)$ for all $m \geq M_n$ and $n \geq N_0$. Thus, 
$$ \limsup_{n \to \infty} \phi(w(\tau^j_n)) \leq \limsup_{m \to \infty}\phi(w(\pi^j_m)) = \phi(w(\sigma_j)) -\pM(\mcK^j).$$
This yields that $ \lim_{n \to \infty} \phi(w(\tau^j_n)) = \phi(w(\sigma_j)) -\pM_k(\mcK^j)$ and hence \eqref{e:cvgDBnj} holds always. 
\end{itemize}
\end{itemize}
\end{proof}

\begin{proof}[Proof of Theorem \ref{thm:clt_msa}]
The proof proceds by verifying the conditions of Theorem \ref{thm:clt_Duy17}. The difficult condition to verify is weak stabilization. Weak stabilization has already been shown to hold in Proposition \ref{prop:stab_death_times}. Since translation invariance holds trivially, we only have to show the fourth moment bound.

Now, all that remains is to verify the Poisson bounded moments condition. Let $0 \in A$ for some $A \in \mathfrak{A}$. From Proposition \ref{prop:bddmomentsA}, we have that there exist random variables $L_A, M_A$ such that
\begin{equation}
\label{e:diff_del_A}
\mcD(\P \cap A \cup \{0\}) \triangle \mcD(\P \cap A) = \{\sigma_1, \ldots, \sigma_{L_A} \} 
\end{equation}
and $w(\sigma_i) \leq M_A$ for all $1 \leq i \leq L_A$. Then by H\"{o}lder's inequality, growth condition on $\phi$ and Theorem \ref{t:stabSTY17}, we obtain
$$ |D_{\0}(\pM(\mcD(\P \cap A)))|^4 \leq 16 L_A M_A^{4p}.$$
The Poisson bounded moment condition now follows by applying Cauchy-Schwarz inequality and the tail bounds for $L_A, M_A$ in Proposition \ref{prop:bddmomentsA}. 

It remains to show that the limiting variances $\sigma^2(\pMi_k,\varphi), \sigma^2(B_k,\varphi)$ and $\sigma^2(L_k,\varphi)$ are non-zero. This is shown in upcoming Proposition \ref{prop:varlb} by proving that these are non-zero random variables. 
\end{proof}

\subsection{Proof of variance lower bound}
\label{s:varlb}

The key to our variance lower bound proof is the existence of a particular robust configuration of points that gives non-trivial lower bounds on the add-one cost for the total weight of minimal spanning acycles.
\begin{lemma}
\label{lem:config}
Let 
$\phi : \mR_+ \to \mR_+$ be a strictly increasing function. For any $r > 0$, there exists $\varepsilon > 0$ (depending on $d$), a finite set of points $\cpts := \cpts(r) = \{x_1,\ldots,x_{2d+2}\} \subset B_{r}(\0)$ and disjoint sets $A(x_i) \subset B_{\epsilon}(x_i)$ with  $|A(x_i)| \geq c\theta_d\epsilon^d$ for some $c > 0$ such that  for any finite point-set $\X$ with $\X \cap B_{3r}(\0) = \{y_1,\ldots,y_{2d+2}\}$ and $y_i\in A(x_i)$ for all $1 \leq i \leq m$,  the following inequalities hold:
\begin{equation}
\label{e:configbounds}
0<\phi(r/4)< D_{\0}(\pM_k(\X)) = \pM_k(\X \cup \{\0\}) - \pM_k(\X)  \quad  \mbox{for}\;\;1\leq k \leq d, \end{equation}
\begin{equation}
\label{e:configbounds_birth}
0<\phi(r/4)< D_{\0}(\pB_k(\X)) = \pB_k(\X \cup \{\0\}) - \pB_k(\X) \quad \mbox{for}\; 1\leq k < d.
\end{equation}
\end{lemma}%
\begin{remark}
\label{r:config}
We note that the proof of the above lemma also implies the result for  $ D_{\0}(\pL_k)$ for $1\leq k<d$. An upper bound of $2^{d+1}\phi(dr)$ also holds for the above the add-one costs - $ D_{\0}(\pM_k(\X)), D_{\0}(\pB_k(\X)),  D_{\0}(\pL_k)$ - and follow easily from our proofs. See Lemma \ref{lem:coarsebound} and proof of Lemma \ref{lem:config}.
\end{remark}
Assuming the above lemma, we prove the required condition for variance lower bound and then work towards the proof of the lemma. Recall $D_{\infty}(\pM_k)$, $D_{\infty}(\pB_k)$, and $D_{\infty}(\pL_k)$ as defined in Proposition \ref{prop:stab_death_times}.  
\begin{prop}
\label{prop:varlb}
Let assumptions be as in Theorem \ref{thm:clt_msa}. The random variables $D_{\infty}(\pM_k)$, $D_{\infty}(\pB_k)$ and $D_{\infty}(\pL_k)$ are non-zero.
\end{prop}
\begin{proof}
We shall use Proposition \ref{prop:stab_death_times} and Lemma \ref{lem:config} to prove that $D_{\infty}(\pM_k)$ is 
non-zero. The same proof may be followed to show that $D_{\infty}(\pB_k)$ and $D_{\infty}(\pL_k)$ are non-zero; see also Remark \ref{r:config}.  So, we shall skip the details for the proofs for $D_{\infty}(\pB_k)$ and $D_{\infty}(\pL_k)$. We shall again drop the subscript $k$.

From Proposition \ref{prop:stab_death_times}, we have that $D_{\0}\pM(\P \cap W_n) \to D_{\infty}(\pM)$ a.s. as $n \to \infty$. Thus, we have that 
\begin{equation}
\label{e:ltDinftyDwn}
\BP(D_{\infty}(\pM) > 0) = \lim_{n \to \infty} \BP(D_{\0}\pM(\P \cap W_n) > 0).
\end{equation}
Since $\phi$ is strictly increasing and non-negative, we have that $\phi(r/4) > 0$ for $r >  0$. Let $n$ be such that $n > 4r$. We shall show that there exists $C_0 > 0$ (depending on $r,\varepsilon,d$) such that for all $n > 4r$
\begin{equation}
\label{e:poissonconfiglb}
\BP(D_{\0}\pM(\P \cap W_n) > 0)  \geq   \BP(D_{\0}\pM(\P \cap W_n) \geq \phi(r/4)) > C_0.
\end{equation}
Thus, we obtain that $D_{\infty}(\pM)$ is a non-zero random variable. Hence, to complete the proof it suffices to show \eqref{e:poissonconfiglb}.

Set $\P_n := \P \cap W_n$. From Lemma \ref{lem:config} and using that $n > 4r$  we have that
\begin{align*}
&  \BP(D_{\0}\pM(\P \cap W_n) \geq \phi(r/4))    \\
& \geq \BP( |\P_n \cap B_{3r}(\0)| = 2d+2,  |\P_n \cap A(x_i)| = 1 \, \forall 1 \leq i \leq 2d+2) \\
& \geq \BP( |\P \cap B_{3r}(\0)| =2d+2,  |\P \cap A(x_i)| = 1 \, \forall 1 \leq i \leq 2d+2) > 0,
\end{align*}
where the last probability is non-zero can be shown using that $|A(x_i)| \geq c\theta_d\epsilon^d, i=1,\ldots,2d+2$ (i.e., $A(x_i)$'s have non-trivial measure) and the two properties of the Poisson process (see Section~\ref{sec:weak_stab}). 
Thus, we have proved \eqref{e:poissonconfiglb} and that $D_{\infty}(\pM)$ is a non-zero random variable as required.
\end{proof}

Now we shall we focus on the proof of Lemma \ref{lem:config}. We first introduce the configuration then prove the required properties. Without loss of generality, we set $r=1$ and $d\geq 2$. The case of MSAs on 1-dimensional spaces is straightforward. Consider $(d+1)$ points which lie the unit $d$-sphere and form a regular $d$-simplex, i.e. all points are mutually equidistant. These form a $d$-simplex which we denote $[p_1,\ldots, p_{d+1}]$. For each $(d-1)$-face, we add a point on the ray from the origin to the center of the face but lying on the hypersphere of radius $\rho$ where $\rho > 20d$. Note that by the center of the face, we are referring to the center of the minimum enclosing ball of the vertices of the face. We denote these points $\{q_1,\ldots, q_{d+1}\}$ such that $q_i$ corresponds to the $(d-1)$-face with $p_i$ removed. 
The configuration for $d=2$ is shown in Figure~\ref{fig:config}. 
\begin{figure}
\centering\includegraphics[width=0.5\textwidth]{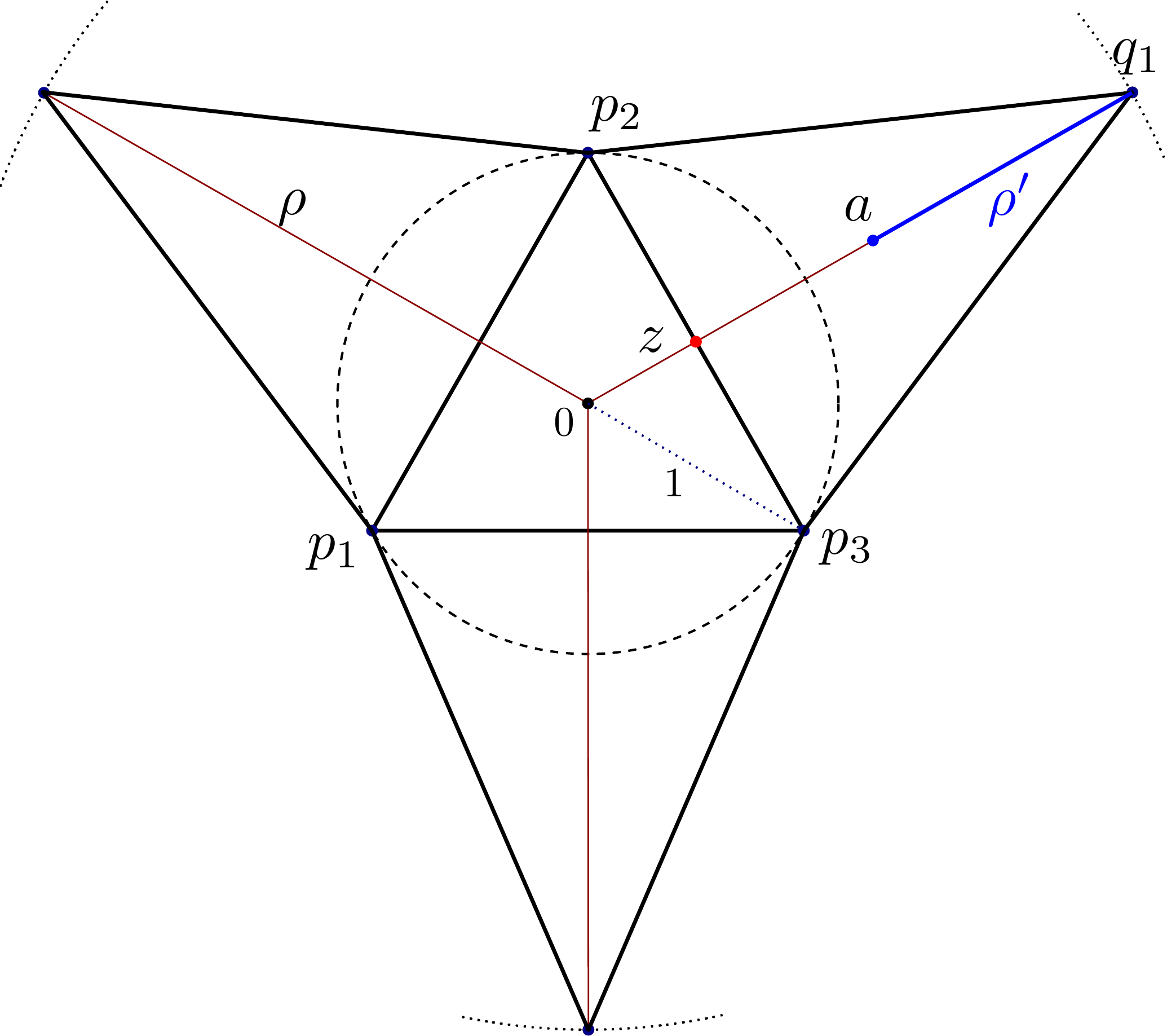}
\caption{\label{fig:config} Configuration of points $\cpts(r)$ in $\mathbb{R}^2$.}
\end{figure}
Setting $\cpts := \cpts(1) = \{p_1,\ldots,p_{d+1}, q_1,\ldots, q_{d+1}\}$, let $\mcD(\cpts)$ denote the Delaunay complex built on $\cpts$ and 
$\mrD(\cpts)$ denote the Delauanay complex built on $\cpts\cup\{\0\}$.
\begin{lemma}\label{lem:geometric1}
Let $\cpts := \cpts(1)$ as above and in $\mcD(\cpts)$, consider the $d$-simplex $\{q_1,\allowbreak p_2,\allowbreak \ldots,\allowbreak p_{d+1}\}$. Let $a$ and $\rho'$ denote the center and radius of its corresponding circumsphere, then 
$$||a|| - \rho' > \frac{1}{10d}.$$ 
The above bound also holds for all $d$-simplices $\{p_1,\ldots,p_{i-1},q_i,p_{i+1},\ldots,p_d\}, i = 1,\ldots,d.$
\end{lemma}

\begin{lemma}\label{lem:coarsebound}
Let $\cpts := \cpts(1)$ be as above and $k \leq d$. Then the following holds. (1) $\mrD(\cpts) \setminus \mcD(\cpts)$ consists of at most $2^{d+1}$ simplices whose maximum weight is $\frac{d}{2}$ and minimum weight is $1/2$.  (2) $\mcD(\cpts) \setminus \mrD(\cpts)$ consists of a single $d$-simplex of weight $1$.
\end{lemma}
\begin{proof}
Lemma ~\ref{lem:geometric1} implies that the $d$-simplices in $\mcD(\cpts)$ are 
\[ [p_1,\ldots,p_{d+1}] \cup \{ [p_1,\ldots,p_{i-1},q_i,p_{i+1},\ldots,p_{d+1}] : i = 1,\ldots,d+1 \} \]
and the $d$-simplices in $\mrD(\cpts)$ are
\[ \{ [\0,p_1,\ldots,p_{i-1},p_{i+1},\ldots,p_{d+1}], [p_1,\ldots,p_{i-1},q_i,p_{i+1},\ldots,p_{d+1}] : i = 1,\ldots,d+1 \}. \]
Note that the two complexes are nothing but the above $d$-simplices and all the lower dimensional ones contained in these $d$-simplices. These are consequences of the fact that the circumspheres of the $d$-simplices are empty in $\cpts$ or $\cpts \cup \{0\}$ as required and hence the $d$-simplices are in $\mcD(\cpts)$ or $\mrD(\cpts)$ as needed.

Thus (2) now follows easily as the only $d$-simplex in $\mrD(\cpts) \setminus \mcD(\cpts)$ is $[p_1,\ldots, p_{d+1}]$ which is weight $1$ by construction.

In $\mrD$, the simplex $[p_1,\ldots, p_{d+1}] \in \mcD$ is replaced by the simplices
\[ \{ [\0,p_{i_1},\ldots,p_{i_k}] : 1 \leq k \leq d,  \{i_1,\ldots,i_k\} \subset \{1,\ldots,d+1\} \}. \]
This gives the upper bound of at most $2^{d+1}$ simplices. The minimum weight is achieved by a $1$-simplex. For example, we may take $[\0,p_i]$ for any $1 \leq i \leq d+1$ and check that $w([\0,p_i]) = 1/2$.

The maximum weight is achieved by any of the $d$-simplices. For example, consider $[\0,p_2,\ldots,p_{d+1}]$.  While all the simplices are contained in the unit sphere, the weight for a simplex in the Delaunay complex may be larger. To bound the new weight we must find where $\0$ and $p_i$ are all equidistant. Since the weight function is known to induce a filtration, the $d$-simplices must have the largest weight. We can directly compute the radius as $\frac{dr}{2}$ using straightforward geometry (see Appendix~\ref{app:weight}).
\end{proof}

\begin{proof}[Proof of Lemma~\ref{lem:config}]
Let $\cpts = \cpts(r) := r\cpts(1)$, where $\cpts(1)$ is  as above. With an abuse of notation, we shall denote the weights on $\mcD(\cpts(r))$ and $\mcD(\cpts(1))$ by $w$. Observe that $[v_0,\ldots,v_k] \in \mcD_{\cpts(1)}$ iff $[rv_0,\ldots,rv_k] \in \mcD_{\cpts(r)}$ and thus trivially
$$w([rv_0,\ldots,rv_k]) = rw([v_0,\ldots,v_k]).$$    
Hence,  the weights in Lemma \ref{lem:coarsebound} when applied to $\cpts(r)$ undergo a simple scaling by $r$.  We first give bounds for $\cpts =\cpts(r)$ and then derive bounds for $\X$ by a perturbation argument. Set $\pM_k(\cpts) = \pM_k(\mcD(\cpts))$ and similarly for $\cpts \cup \{\0\}$ and $\pB_k$. 

We consider two cases. First for $k<d$, $k$-simplices are only added. While the weights of the simplices which are in both complexes may change, the weights may only increase. 
Since we add at least one $k$-simplex with weight at least $\frac{r}{2}$ and the added edges have weight $\frac{r}{2}$, all other simplices must have higher weight. Also we observe that we add at least one positive and at least one negative simplex for $k<d$ (see Lemma \ref{l:bettideath}). Using Theorem \ref{t:birthdeathmsa}, this gives the following bound for $k<d$.
\begin{equation}
\label{e:pMkQlowerbd}
 \pM_k(\cpts\cup \{\0\}) - \pM_k(\cpts) \geq \phi\left(\frac{r}{2}\right) \, ; \,
  \pB_k(\cpts \cup \{\0\}) - \pB_k(\cpts) \geq \phi\left(\frac{r}{2}\right).
 \end{equation}

For $k=d$, we need only consider $\pM_d$, since there are no births in dimension $d$. We note that we add $(d+1)$ $d$-simplices, one for each $(d-1)$-face, each of weight $\frac{dr}{2}$ and we remove one $d$-simplex of weight $r$. This gives the lower bounds  
\begin{align}
\label{e:pMdQlowerbd}
\begin{split}
    \pM_d(\cpts\cup \{\0\}) - \pM_d(\cpts) &\geq (d+1)\phi\left(\frac{dr}{2}\right)-\phi(r)\\
    &\geq d\phi\left(\frac{dr}{2}\right)\geq \phi\left(\frac{r}{2}\right) >0
    \end{split}
\end{align}
The final inequalities follow since $\phi$ is a strictly increasing function and $d>1$. 

We extend this to $\X$ as in statement of the lemma by a perturbation argument for $\varepsilon > 0$ small enough. For  $d\geq 3$,  choose $\varepsilon > 0$ such that
\begin{equation}\label{eq:epsbound}
 2(d+2)\varepsilon < \frac{r}{4}, 
 \end{equation}
%
Now let $\X$ be as in the statement of the lemma and assume that $\cpts' := \X \cap B_{3r}(\0) = \{p'_1,\ldots,p'_{d+1},q'_1,\ldots,q'_{d+1}\}$ where $|p_i-p'_i|, |q_i - q'_i| \leq \varepsilon$ for all $i =1,\ldots,d+1$, i.e.  $ A(p_i) = B_\varepsilon(p_i)$ and $A(q_i) = B_\varepsilon(q_i)$. 

Observe that there is a bijection between simplices in $\mcD_{\cpts'}$ (resp. $\mrD_{\cpts}$) and $\mcD_{\cpts}$ (resp. $\mcD_{\cpts'}$) with the difference between the corresponding weights bounded above by $2(d+2)\varepsilon$. Thus by assumption on $\X$ and construction of $\cpts$, we have that
\[ \pM_k(\X \cup \{\0\}) - \pM_k(\X) = \pM_k(\cpts' \cup \{\0\}) - \pM_k(\cpts'), \]
and similarly for $\pB_k$. Now following the derivation as in \eqref{e:pMkQlowerbd} and using our choice of $\varepsilon$,  we derive that for $1 \leq k < d$ 
\[ 
\phi(r/4) < \pM_k(\mathcal{X}\cup \{\0\}) - \pM_k(\mathcal{X}), 
\]
\[
\phi(r/4) < \pB_k(\mathcal{X}\cup \{\0\}) - \pB_k(\mathcal{X}), \]
Using \eqref{e:pMdQlowerbd} in the above argument, we can also derive for $k = d$,  $d > 3$ that
\[ \phi(r/4) < \pM_k(\mathcal{X}\cup \{\0\}) - \pM_k(\mathcal{X})  .\]
These prove \eqref{e:configbounds_birth} with $A(x_i) = B_{\varepsilon}(x_i)$. Thus, the above arguments have shown \eqref{e:configbounds} for $1 \leq k < d, d \geq 2$ and $k = d \geq 3$ with $A(x_i) = B_{\varepsilon}(x_i)$. This leaves the case of $k = d = 2$ open in \eqref{e:configbounds}.  

Let $k = d = 2$.  As we only require $\phi$ to be a positive increasing function, the above perturbation argument does not work for $d=2$  as is for the lower bound. In \eqref{e:pMdQlowerbd}, it is sufficient to show that after perturbation, at least one of the added simplices has a weight higher than the removed simplex.  Once this is established, the rest of the proof remains the same. 
As in the higher dimensional case, we must choose $\varepsilon > 0$ small enough so that the combinatorial structure of the complex changes in a predictable way. Therefore, as in \eqref{eq:epsbound}, we can choose
%
$\varepsilon < \frac{r}{40}$,
and set $A(q_i)=B_\varepsilon(q_i)$. 
Let a cone at a point $x$ in the direction $\overrightarrow{v}$ of angle $\alpha$ be denoted by $\mathrm{Cone}(x,\overrightarrow{v},\alpha)$. We the set:
\begin{align*}
    A(p_1) &= B_\varepsilon(p_1)\cap \mathrm{Cone}(p_1,(-2,-1),\pi/12)\\
    A(p_2) &= B_\varepsilon(p_2)\cap \mathrm{Cone}(p_2,(0,-1),\pi/12)\\
    A(p_3) &= B_\varepsilon(p_3)\cap \mathrm{Cone}(p_3,(2,-1),\pi/12)
\end{align*}
This can be seen in Figure~\ref{fig:cones}.
\begin{figure}
    \centering
    \includegraphics[width=0.5\textwidth,page=6]{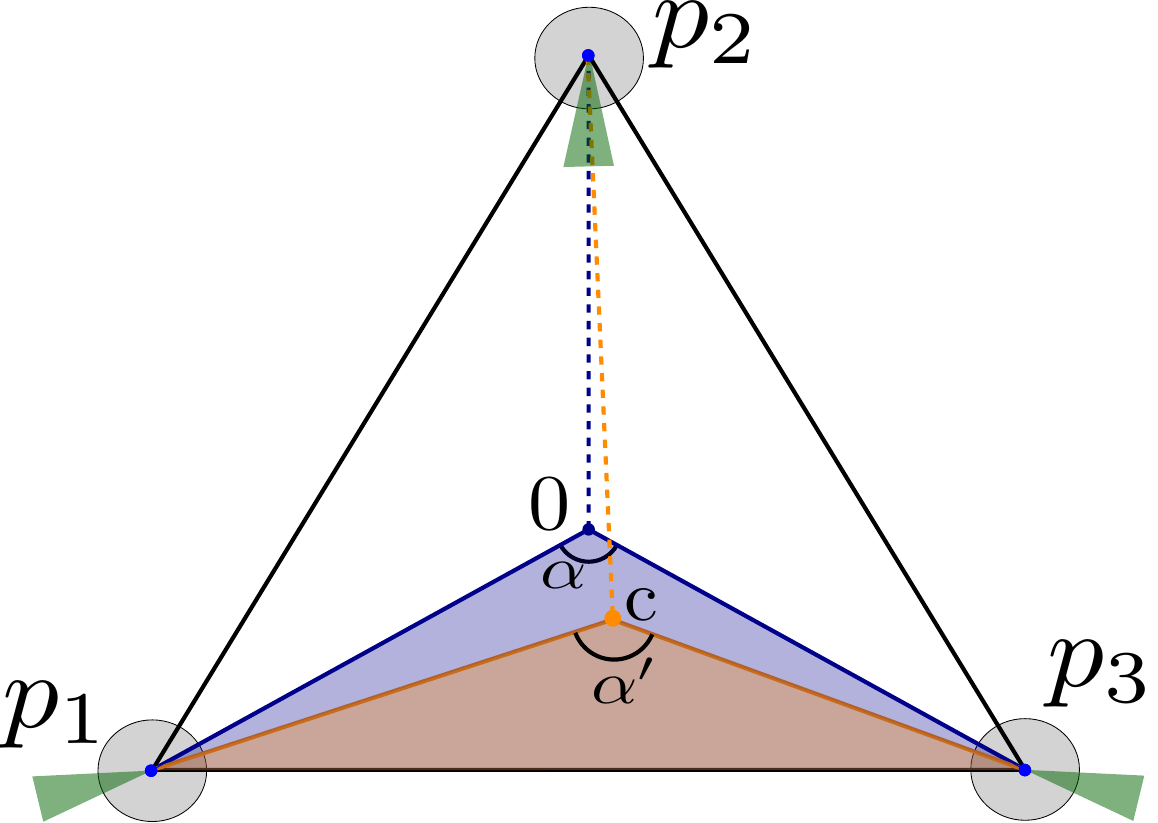}
    \caption{Balls around the points in the configuration for $d=2$ are shown in gray. The corresponding cones $A(p_i)$ are shown in green.}
    \label{fig:cones}
\end{figure}
Let $p'_i\in A(p_i)$ for all $i$. To complete the proof of bounds in \eqref{e:configbounds} and \eqref{e:configbounds_birth}, it suffices to show that
\[ w([p'_1,p'_2,p'_3]) \leq w([p'_1,\0,p'_3]) .\]
%

We observe that the angle $\angle (p'_1, \0, p'_3) \geq \frac{2\pi}{3}$, which we denote by $\alpha$. Further, let $c$ denote the center of the circumcircle of the points $\{p'_1,p'_2,p'_3\}$. It is  straightforward to also check that the angle $\angle (p'_1, c, p'_3)\geq \frac{2\pi}{3}$, which we denote by $\alpha'$.   

The weight  $w([p'_1,p'_2,p'_3])$ can be related with $||\overline{p'_1 p'_3}||$ and $\alpha$.
\begin{equation*}
 ||\overline{p'_1 p'_3}||^2  =  2 w^2([p'_1,p'_2,p'_3]) (1-\cos \alpha')
\end{equation*}
Likewise, the $w([p'_1,\0,p'_3]$ may be expressed as
\begin{align*}
    w([p'_1,\0,p'_3] = \frac{ ||\overline{p'_1 p'_3}||}{2\sin\alpha}
\end{align*}
Combining these we obtain
\begin{align*}
    \frac{w^2([p'_1,\0,p'_3]}{w^2([p'_1,p'_2,p'_3])} = \frac{ (1-\cos \alpha')}{2\sin^2\alpha}
\end{align*}
Rewriting $\alpha=2\pi/3 +\delta$ and $\alpha'=2\pi/3 +\delta'$ for $\delta,\delta'>0$, we must show that
\begin{align*}
    \frac{ (1-\cos(2\pi/3+\delta'))}{2\sin^2(2\pi/3+\delta)} \geq 1
\end{align*}
We note that $\delta'\geq \delta$, which are functions of $\varepsilon$, where $\delta' \rightarrow 0$ as $\varepsilon\rightarrow 0$. So as $\delta'\rightarrow 0$, we can directly verify that this expression approaches 1 from above proving the result for $\varepsilon$ small enough. Using the triangle inequality on the points, one can directly verify that the above inequality holds for $\varepsilon<\frac{r}{40}$. As stated above, this implies that the weight of at least one added simplex is greater that the removed simplex, so the lower bound of $\phi(r/2)$ holds.


%
%
\end{proof}
\begin{remark}
\label{r:configdc}
The construction above does not hold for all $\phi$ if we are using the Delaunay-\v Cech weights rather than  Alpha  weights. Rather, the configuration must be changed to achieve similar bounds. The alternate configuration is moving one of the inner points to the origin (and suitable adjusting the outer points), resulting in a flattened $d$-simplex, which is translated so that the centroid of the simplex is the origin. For this case, the upper bound as stated holds, but the lower bound can be shown to be strictly positive and depending on $r$ for all $\phi$.   
\end{remark}

\section*{Acknowledgements}
DY's research was partially supported by SERB-MATRICS Grant MTR/2020/000470, DST-INSPIRE Faculty grant and CPDA from the Indian Statistical Institute. The project also benefitted from discussion between the authors during their visit to Banff International Research Station in August 2017 and would like to thank the institution for its support. The authors would like to thank Omer Bobrowski and Gugan Thoppe for discussions at an early stage of the project. DY is very thankful to K. D. Trinh for pointing out some simplifications regarding growth conditions and variance lower bound using results from \cite{trinh2019central,trinh2022random}. DY is thankful to Xiaochuan Yang for some comments on an earlier draft.  

\bibliographystyle{plainnat}
\bibliography{CLT_MSA}

\begin{thebibliography}{45}
\providecommand{\natexlab}[1]{#1}
\providecommand{\url}[1]{\texttt{#1}}
\expandafter\ifx\csname urlstyle\endcsname\relax
  \providecommand{\doi}[1]{doi: #1}\else
  \providecommand{\doi}{doi: \begingroup \urlstyle{rm}\Url}\fi

\bibitem[Aldous and Steele(1992)]{aldous1992asymptotics}
D.~Aldous and J.~M. Steele.
\newblock Asymptotics for {E}uclidean minimal spanning trees on random points.
\newblock \emph{Probability Theory and Related Fields}, 92\penalty0
  (2):\penalty0 247--258, 1992.

\bibitem[Alexander(1996)]{Alexander1995}
K.~S. Alexander.
\newblock The {RSW} theorem for continuum percolation and the {CLT} for
  {E}uclidean minimal spanning trees.
\newblock \emph{The Annals of Applied Probability}, 6\penalty0 (2):\penalty0
  466--494, 1996.

\bibitem[Baccelli et~al.(2020)Baccelli, B{\l}aszczyszyn, and
  Karray]{baccelli2020random}
F.~Baccelli, B.~B{\l}aszczyszyn, and M.~Karray.
\newblock Random measures, point processes, and stochastic geometry.
\newblock Lecture Notes, 2020.
\newblock URL \url{https://hal.inria.fr/hal-02460214}.

\bibitem[Bauer and Edelsbrunner(2017)]{bauer2017morse}
U.~Bauer and H.~Edelsbrunner.
\newblock The {M}orse theory of \uppercase{\v{c}}ech and {D}elaunay complexes.
\newblock \emph{Transactions of the American Mathematical Society},
  369\penalty0 (5):\penalty0 3741--3762, 2017.

\bibitem[Bj{\"o}rner(1995)]{bjorner1995topological}
A.~Bj{\"o}rner.
\newblock Topological methods.
\newblock \emph{Handbook of combinatorics}, 2:\penalty0 1819--1872, 1995.

\bibitem[B{\l}aszczyszyn et~al.(2019)B{\l}aszczyszyn, Yogeshwaran, and
  Yukich]{BYY2019}
B.~B{\l}aszczyszyn, D.~Yogeshwaran, and J.~E. Yukich.
\newblock Limit theory for geometric statistics of point processes having fast
  decay of correlations.
\newblock \emph{The Annals of Probability}, 47\penalty0 (2):\penalty0 835--895,
  2019.

\bibitem[Can and Trinh(2022)]{trinh2022random}
V.~H. Can and K.~D. Trinh.
\newblock Random connection models in the thermodynamic regime: central limit
  theorems for add-one cost stabilizing functionals.
\newblock \emph{Electronic Journal of Probability}, 27:\penalty0 1--40, 2022.

\bibitem[Cao(2021)]{cao2021central}
S.~Cao.
\newblock Central limit theorems for combinatorial optimization problems on
  sparse {E}rd{\H{o}}s--{R}{\'e}nyi graphs.
\newblock \emph{The Annals of Applied Probability}, 31\penalty0 (4):\penalty0
  1687--1723, 2021.

\bibitem[Carlsson(2009)]{Carlsson09}
G.~Carlsson.
\newblock Topology and data.
\newblock \emph{Bulletin of the Ameriancan Mathematical Society}, 46\penalty0
  (2):\penalty0 255--308, 2009.

\bibitem[Carlsson(2014)]{Carlsson14}
G.~Carlsson.
\newblock Topological pattern recognition for point cloud data.
\newblock \emph{Acta Numerica}, 23:\penalty0 289--368, 2014.

\bibitem[Chatterjee and Sen(2017)]{chatterjee2017minimal}
S.~Chatterjee and S.~Sen.
\newblock Minimal spanning trees and {S}tein's method.
\newblock \emph{The Annals of Applied Probability}, 27\penalty0 (3):\penalty0
  1588--1645, 2017.

\bibitem[Cohen-Steiner et~al.(2007)Cohen-Steiner, Edelsbrunner, and
  Harer]{cohen2007stability}
D.~Cohen-Steiner, H.~Edelsbrunner, and J.~Harer.
\newblock Stability of persistence diagrams.
\newblock \emph{Discrete and Computational Geometry}, 37\penalty0 (1):\penalty0
  103--120, 2007.

\bibitem[Delfinado and Edelsbrunner(1993)]{delfinado1993incremental}
C.~J.~A. Delfinado and H.~Edelsbrunner.
\newblock An incremental algorithm for {B}etti numbers of simplicial complexes.
\newblock In \emph{Proceedings of the ninth annual Symposium on Computational
  geometry}, pages 232--239. ACM, 1993.

\bibitem[Divol and Polonik(2019)]{divol2019choice}
V.~Divol and W.~Polonik.
\newblock On the choice of weight functions for linear representations of
  persistence diagrams.
\newblock \emph{Journal of Applied and Computational Topology}, 3\penalty0
  (3):\penalty0 249--283, 2019.

\bibitem[Edelsbrunner and Harer(2010)]{Edelsbrunner10}
H.~Edelsbrunner and J.~L. Harer.
\newblock \emph{Computational Topology, An Introduction}.
\newblock American Mathematical Society, Providence, RI, 2010.

\bibitem[Edelsbrunner and {\"O}lsb{\"o}ck(2020)]{edelsbrunner2020tri}
H.~Edelsbrunner and K.~{\"O}lsb{\"o}ck.
\newblock Tri-partitions and bases of an ordered complex.
\newblock \emph{Discrete \& Computational Geometry}, 64\penalty0 (3):\penalty0
  759--775, 2020.

\bibitem[Hatcher(2002)]{Hatcher02}
A.~Hatcher.
\newblock \emph{Algebraic Topology}.
\newblock Cambridge University Press, Cambridge, New York, 2002.

\bibitem[Hino and Kanazawa(2019)]{hino2019asymptotic}
M.~Hino and S.~Kanazawa.
\newblock Asymptotic behavior of lifetime sums for random simplicial complex
  processes.
\newblock \emph{Journal of the Mathematical Society of Japan}, 71\penalty0
  (3):\penalty0 765--804, 2019.

\bibitem[Hiraoka and Shirai(2017)]{hiraoka2017minimum}
Y.~Hiraoka and T.~Shirai.
\newblock Minimum spanning acycle and lifetime of persistent homology in the
  {L}inial--{M}eshulam process.
\newblock \emph{Random Structures and Algorithms}, 51\penalty0 (2):\penalty0
  315--340, 2017.

\bibitem[Hiraoka and Tsunoda(2018)]{Hiraoka2018cub}
Y.~Hiraoka and K.~Tsunoda.
\newblock Limit theorems for random cubical homology.
\newblock \emph{Discrete and Computational Geometry}, 60\penalty0 (3):\penalty0
  665--687, 2018.

\bibitem[Hiraoka et~al.(2018)Hiraoka, Shirai, and Trinh]{Hiraoka2018limit}
Y.~Hiraoka, T.~Shirai, and K.~D. Trinh.
\newblock Limit theorems for persistence diagrams.
\newblock \emph{The Annals of Applied Probability}, 28\penalty0 (5):\penalty0
  2740--2780, 2018.

\bibitem[Janson(1995)]{janson1995minimal}
S.~Janson.
\newblock The minimal spanning tree in a complete graph and a functional limit
  theorem for trees in a random graph.
\newblock \emph{Random Structures and Algorithms}, 7\penalty0 (4):\penalty0
  337--355, 1995.

\bibitem[Kalai(1983)]{kalai1983enumeration}
G.~Kalai.
\newblock Enumeration of $\mathbb{Q}$-acyclic simplicial complexes.
\newblock \emph{Israel Journal of Mathematics}, 45\penalty0 (4):\penalty0
  337--351, 1983.

\bibitem[Kanazawa(2021)]{kanazawa2021law}
S.~Kanazawa.
\newblock Law of large numbers for {B}etti numbers of homogeneous and spatially
  independent random simplicial complexes.
\newblock \emph{Random Structures and Algorithms}, 2021.

\bibitem[Kesten and Lee(1996)]{kesten1996central}
H.~Kesten and S.~Lee.
\newblock The central limit theorem for weighted minimal spanning trees on
  random points.
\newblock \emph{The Annals of Applied Probability}, pages 495--527, 1996.

\bibitem[Krebs and Polonik(2019)]{krebs2019asymptotic}
J.~T.~N. Krebs and W.~Polonik.
\newblock On the asymptotic normality of persistent {B}etti numbers.
\newblock \emph{arXiv:1903.03280}, 2019.

\bibitem[Lachi{\`e}ze-Rey et~al.(2020)Lachi{\`e}ze-Rey, Peccati, and
  Yang]{lachieze2020quantitative}
R.~Lachi{\`e}ze-Rey, G.~Peccati, and X.~Yang.
\newblock Quantitative two-scale stabilization on the {P}oisson space.
\newblock \emph{arXiv:2010.13362}, 2020.

\bibitem[Last and Penrose(2017)]{last2017lectures}
G.~Last and M.~Penrose.
\newblock \emph{Lectures on the {P}oisson Process}, volume~7.
\newblock Cambridge University Press, 2017.

\bibitem[Lee(1997)]{lee1997central}
S.~Lee.
\newblock The central limit theorem for {E}uclidean minimal spanning trees {I}.
\newblock \emph{The Annals of Applied Probability}, 7\penalty0 (4):\penalty0
  996--1020, 1997.

\bibitem[Lee(1999)]{lee1999central}
S.~Lee.
\newblock The central limit theorem for {E}uclidean minimal spanning trees
  {II}.
\newblock \emph{Advances in Applied Probability}, 31\penalty0 (4):\penalty0
  969--984, 1999.

\bibitem[McGivney and Yukich(1999)]{McGivney99}
K.~McGivney and J.~E. Yukich.
\newblock Asymptotics for {V}oronoi tessellations on random samples.
\newblock \emph{Stochastic Process. Appl.}, 83\penalty0 (2):\penalty0 273--288,
  1999.

\bibitem[Munkres(1984)]{Munkres84}
J.~R. Munkres.
\newblock \emph{Elements of Algebraic Topology}.
\newblock Addison-Wesley, 1984.

\bibitem[Penrose(2003)]{Penrose03}
M.~D. Penrose.
\newblock \emph{Random geometric graphs}, volume~5 of \emph{Oxford Studies in
  Probability}.
\newblock Oxford University Press, Oxford, 2003.

\bibitem[Penrose(2007)]{Penrose2007laws}
M.~D. Penrose.
\newblock Laws of large numbers in stochastic geometry with statistical
  applications.
\newblock \emph{Bernoulli}, 13\penalty0 (4):\penalty0 1124--1150, 2007.

\bibitem[Penrose and Yukich(2001)]{Penrose01}
M.~D. Penrose and J.~E. Yukich.
\newblock Central limit theorems for some graphs in computational geometry.
\newblock \emph{The Annals of Applied Probability}, 11\penalty0 (4):\penalty0
  1005--1041, 2001.

\bibitem[Penrose and Yukich(2002)]{penrose2002limit}
M.~D. Penrose and J.~E. Yukich.
\newblock Limit theory for random sequential packing and deposition.
\newblock \emph{The Annals of Applied Probability}, 12\penalty0 (1):\penalty0
  272--301, 2002.

\bibitem[Penrose and Yukich(2003)]{penrose2003weak}
M.~D. Penrose and J.~E. Yukich.
\newblock Weak laws of large numbers in geometric probability.
\newblock \emph{The Annals of Applied Probability}, 13\penalty0 (1):\penalty0
  277--303, 2003.

\bibitem[Ramey(1983)]{ramey1983non}
D.~B. Ramey.
\newblock A non-parametric test of bimodality with applications to cluster
  analysis.
\newblock Technical report, NASA, 1983.

\bibitem[Skraba and Bobrowski(2020)]{skraba2020homological}
P.~Skraba and O.~Bobrowski.
\newblock Homological percolation: The formation of giant k-cycles.
\newblock \emph{International Mathematics Research Notices}, 2020.

\bibitem[{Skraba} et~al.(2020){Skraba}, {Thoppe}, and {Yogeshwaran}]{Skraba17}
P.~{Skraba}, G.~{Thoppe}, and D.~{Yogeshwaran}.
\newblock Randomly weighted $d-$complexes: Minimal spanning acycles and
  persistence diagrams.
\newblock \emph{The Electronic Journal of Combinatorics}, 27\penalty0
  (2):\penalty0 P2.11, 2020.

\bibitem[Steele(1997)]{Steele97}
J.~M. Steele.
\newblock \emph{Probability theory and combinatorial optimization}, volume~69
  of \emph{CBMS-NSF Regional Conference Series in Applied Mathematics}.
\newblock SIAM, Philadelphia, 1997.

\bibitem[Trinh(2019)]{trinh2019central}
K.~D. Trinh.
\newblock On central limit theorems in stochastic geometry for add-one cost
  stabilizing functionals.
\newblock \emph{Electronic Communications in Probability}, 24, 2019.

\bibitem[Yogeshwaran et~al.(2017)Yogeshwaran, Subag, and Adler]{Yogesh17}
D.~Yogeshwaran, E.~Subag, and R.~J. Adler.
\newblock Random geometric complexes in the thermodynamic regime.
\newblock \emph{Probability Theory and Related Fields}, 167\penalty0
  (1):\penalty0 107--142, 2017.

\bibitem[Yukich(1998)]{Yukich98}
J.~E. Yukich.
\newblock \emph{Probability Theory of Classical {E}uclidean Optimization
  problems}, volume 1675 of \emph{Lecture Notes in Mathematics}.
\newblock Springer-Verlag, Berlin, 1998.

\bibitem[Zomorodian and Carlsson(2005)]{zomorodian2005computing}
A.~Zomorodian and G.~Carlsson.
\newblock Computing persistent homology.
\newblock \emph{Discrete and Computational Geometry}, 33\penalty0 (2):\penalty0
  249--274, 2005.

\end{thebibliography}

\appendix
\section{Some geometric computations.}\label{app:weight}
\begin{proof}[Proof of Lemma \ref{lem:geometric1}]
This follows from elementary geometry. Let $z$
be the center point of the simplex $[p_2,\ldots,p_{d+1}]$, where the ray from the origin to $q_1$ intersects the simplex. See Figure~\ref{fig:config} for an illustration. We note that
\begin{align*} 
||z||=\frac{1}{d}, \qquad\qquad ||\overline{zp_i}|| = \frac{\sqrt{d^2-1}}{d},  \, i = 2,\ldots,d+1. 
\end{align*}
We may rewrite:
$$||a|| - \rho' = ||\overline{zq_1}|| + ||z|| - 2\rho' = ||\overline{zq_1}|| + \frac{1}{d} - 2\rho' $$ 
We set $||\overline{zq_1}|| \geq 10d$, so that $\rho \geq 20d \geq 10d+\frac{1}{d}$. Let $\alpha$ denote the angle of $z q_1 p_2$. Then we have
$$ \cos \alpha = \frac{||\overline{zq_1}||}{||\overline{p_2q_1}||}, \qquad \rho' = \frac{||\overline{p_2q_1}||}{2\cos \alpha}$$
Combining these we obtain, 
$$2 \rho' = ||\overline{zq_1}|| + \frac{d^2-1}{d^2 ||\overline{zq_1}||} $$
Substituting above 
\begin{align*}
||a|| - \rho' &= ||\overline{zq_1}|| + \frac{1}{d} - 2\rho'\\ 
&= \frac{1}{d} - \frac{d^2-1}{d^2 ||\overline{zq_1}||} 
\end{align*}
Setting $||\overline{zq_1}||\geq 10d$, 
\begin{align*}
||a|| - \rho' &\geq \frac{1}{d} - \frac{d^2-1}{10d^3}\\
& = \frac{9d^2+1}{10d^3} \geq \frac{1}{10d}
\end{align*}
yielding the result.
\end{proof}

We give a derivation of the geometric fact used in the proof of Lemma \ref{lem:coarsebound}. \\

\paragraph{{\bf Weight of added $d$-simplex in $\cpts(r)$: $w([0,p_2,\ldots,p_{d+1}] = dr/2$. \\ \\}} 

Here we compute the weight of the added $d$-simplices in the Delaunay complex for  the  configuration $\cpts(r)$ in Section~\ref{s:varlb}. Without loss of generality, we can consider  the simplex $\sigma=[\0,p_1,p_2,\ldots,p_d]$. 
%
%
We now compute $w = w(\sigma)$. Let $z$ be the center of $\tau$, and $s$ denote the circumcenter of $\{p_1,p_2,\ldots,p_d]\}$ and $w$ the radius of the circumsphere (see Figure~\ref{fig:small} below). 
\begin{figure}[h!]
    \centering
    \includegraphics[width=0.5\textwidth, page=2]{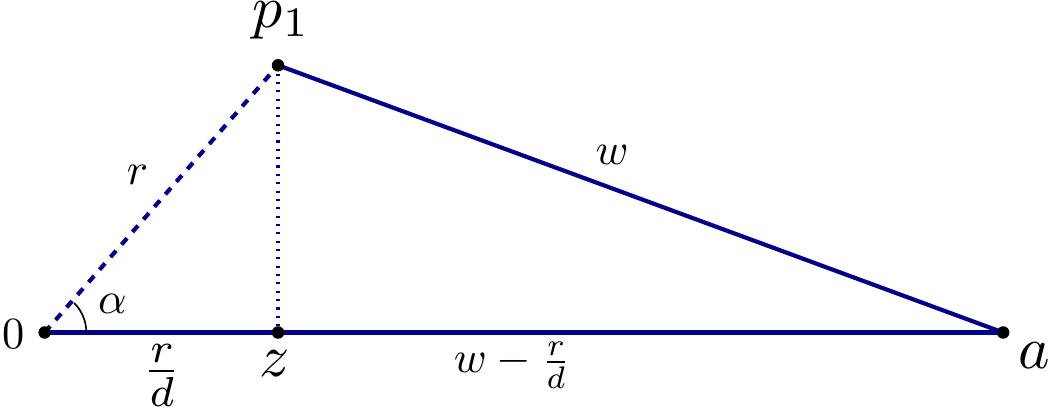}
    \caption{An illustration of how the weight of the added simplex, denoted $w$, is computed, i.e. the distance to the circumcenter $a$.}
    \label{fig:small}
\end{figure}
We have the following identity $\cos \alpha = 1/d$.
Applying the cosine rule:
\begin{align*}
    w^2  = r^2 +w^2 -2wr \cos \alpha\quad
   \Rightarrow \quad w = \frac{dr}{2}
\end{align*} 
we obtain the result. 
\end{document}